\pdfoutput=1
\documentclass[11pt, a4paper, twoside,leqno]{amsart}
\usepackage[centering, totalwidth = 380pt, totalheight = 600pt]{geometry}
\usepackage{amssymb, amsmath, amsthm, thmtools, thm-restate}
\usepackage{microtype, stmaryrd, url,lmodern,eucal}
\usepackage[shortlabels]{enumitem}
\usepackage[latin1]{inputenc}
\usepackage{color}
\definecolor{darkgreen}{rgb}{0,0.45,0}
\usepackage[colorlinks,citecolor=darkgreen,linkcolor=darkgreen]{hyperref}
\usepackage[british]{babel}
\makeatletter
\def\@cite#1#2{[{#1\if@tempswa ,~#2\fi}]}
\makeatother

\DeclareMathAlphabet{\mathbf}{OT1}{cmr}{b}{n}

\usepackage[arrow, matrix, tips, curve, graph, rotate]{xy}
\SelectTips{cm}{10}

\makeatletter
\def\matrixobject@{%
  \edef \next@{={\DirectionfromtheDirection@ }}%
  \expandafter \toks@ \next@ \plainxy@
  \let\xy@@ix@=\xyq@@toksix@
  \xyFN@ \OBJECT@}
\let\xy@entry@@norm=\entry@@norm
\def\entry@@norm@patched{%
  \let\object@=\matrixobject@
  \xy@entry@@norm }
\AtBeginDocument{\let\entry@@norm\entry@@norm@patched}
\makeatother

\newcommand{\twocong}[2][0.5]{\ar@{}[#2] \save ?(#1)*{\cong}\restore}
\newcommand{\twosim}[2][0.5]{\ar@{}[#2] \save ?(#1)*{\simeq}\restore}
\newcommand{\twoeq}[2][0.5]{\ar@{}[#2] \save ?(#1)*{=}\restore}
\newcommand{\rtwocell}[3][0.5]{\ar@{}[#2] \ar@{=>}?(#1)+/l 0.2cm/;?(#1)+/r 0.2cm/^{#3}}
\newcommand{\ltwocell}[3][0.5]{\ar@{}[#2] \ar@{=>}?(#1)+/r 0.2cm/;?(#1)+/l 0.2cm/^{#3}}
\newcommand{\ltwocello}[3][0.5]{\ar@{}[#2] \ar@{=>}?(#1)+/r 0.2cm/;?(#1)+/l 0.2cm/_{#3}}
\newcommand{\dtwocell}[3][0.5]{\ar@{}[#2] \ar@{=>}?(#1)+/u  0.2cm/;?(#1)+/d 0.2cm/^{#3}}
\newcommand{\dltwocell}[3][0.5]{\ar@{}[#2] \ar@{=>}?(#1)+/ur  0.2cm/;?(#1)+/dl 0.2cm/^{#3}}
\newcommand{\drtwocell}[3][0.5]{\ar@{}[#2] \ar@{=>}?(#1)+/ul  0.2cm/;?(#1)+/dr 0.2cm/^{#3}}
\newcommand{\dthreecell}[3][0.5]{\ar@{}[#2] \ar@3{->}?(#1)+/u  0.2cm/;?(#1)+/d 0.2cm/^{#3}}
\newcommand{\utwocell}[3][0.5]{\ar@{}[#2] \ar@{=>}?(#1)+/d 0.2cm/;?(#1)+/u 0.2cm/_{#3}}
\newcommand{\dtwocelltarg}[3][0.5]{\ar@{}#2 \ar@{=>}?(#1)+/u  0.2cm/;?(#1)+/d 0.2cm/^{#3}}
\newcommand{\utwocelltarg}[3][0.5]{\ar@{}#2 \ar@{=>}?(#1)+/d  0.2cm/;?(#1)+/u 0.2cm/_{#3}}

\newcommand{\ulthreecell}[3][0.5]{\ar@{}[#2] \ar@3?(#1)+/dr  0.2cm/;?(#1)+/ul 0.2cm/_{#3}}

\newcommand{\pullbackcorner}[1][dr]{\save*!/#1-1.2pc/#1:(-1,1)@^{|-}\restore}

\newdir{n}{{}*!<-.14em,0em>-\cir<.14em>{u^d}}
\newdir{(}{{}*!<0em,-.14em>-\cir<.14em>{l^r}}
\newdir{ (}{{}*!/-5pt/\dir{(}}
\newdir{ n}{{}*!/-5pt/\dir{n}}
\newdir{ >}{{}*!/-5pt/\dir{>}}

\DeclareMathOperator{\el}{el}

\DeclareMathOperator{\colim}{colim}

\DeclareMathOperator{\im}{im}

\newcommand{\cat}[1]{\mathrm{\mathcal #1}}

\newcommand{\thg}{{\mathord{\text{--}}}}

\newcommand{\abs}[1]{{\left|{#1}\right|}}

\newcommand{\res}[2]{\left.{#1}\right|_{#2}}

\newcommand{\spn}[1]{{\langle{#1}\rangle}}

\newcommand{\defeq}{\mathrel{\mathop:}=}
\newcommand{\cd}[2][]{\vcenter{\hbox{\xymatrix#1{#2}}}}

\renewcommand{\phi}{\varphi}
\newcommand{\A}{{\mathcal A}}

\newcommand{\C}{{\mathcal C}}
\newcommand{\D}{{\mathcal D}}
\newcommand{\E}{{\mathcal E}}
\newcommand{\F}{{\mathcal F}}

\newcommand{\I}{{\mathcal I}}

\newcommand{\M}{{\mathcal M}}

\renewcommand{\P}{{\mathcal P}}

\let\sec=\S
\renewcommand{\S}{{\mathcal S}}

\newcommand{\U}{{\mathcal U}}
\newcommand{\V}{{\mathcal V}}
\newcommand{\W}{{\mathcal W}}

\newcommand{\xtor}[1]{\cdl[@1]{{} \ar[r]|-{\object@{|}}^{#1} & {}}}

\makeatletter

\def\hookleftarrowfill@{\arrowfill@\leftarrow\relbar{\relbar\joinrel\rhook}}
\def\twoheadleftarrowfill@{\arrowfill@\twoheadleftarrow\relbar\relbar}
\def\leftbararrowfill@{\arrowdoublefill@{\leftarrow\mkern-5mu}\relbar\mapstochar\relbar\relbar}
\def\Leftbararrowfill@{\arrowdoublefill@{\Leftarrow\mkern-2mu}\Relbar\Mapstochar\Relbar\Relbar}
\def\leftringarrowfill@{\arrowdoublefill@{\leftarrow\mkern-3mu}\relbar{\mkern-3mu\circ\mkern-2mu}\relbar\relbar}
\def\lefttriarrowfill@{\arrowfill@{\mathrel\triangleleft\mkern0.5mu\joinrel\relbar}\relbar\relbar}
\def\Lefttriarrowfill@{\arrowfill@{\mathrel\triangleleft\mkern1mu\joinrel\Relbar}\Relbar\Relbar}
\def\rightarrowtailfill@{\arrowfill@{\Yright\joinrel\relbar}\relbar\rightarrow}

\def\hookrightarrowfill@{\arrowfill@{\lhook\joinrel\relbar}\relbar\rightarrow}
\def\twoheadrightarrowfill@{\arrowfill@\relbar\relbar\twoheadrightarrow}
\def\rightbararrowfill@{\arrowdoublefill@{\relbar\mkern-0.5mu}\relbar\mapstochar\relbar\rightarrow}
\def\Rightbararrowfill@{\arrowdoublefill@{\Relbar\mkern-2mu}\Relbar\Mapstochar\Relbar\Rightarrow}
\def\rightringarrowfill@{\arrowdoublefill@\relbar\relbar{\mkern-2mu\circ\mkern-3mu}\relbar{\mkern-3mu\rightarrow}}
\def\righttriarrowfill@{\arrowfill@\relbar\relbar{\relbar\joinrel\mkern0.5mu\mathrel\triangleright}}
\def\Righttriarrowfill@{\arrowfill@\Relbar\Relbar{\Relbar\joinrel\mkern1mu\mathrel\triangleright}}

\def\leftrightarrowfill@{\arrowfill@\leftarrow\relbar\rightarrow}
\def\mapstofill@{\arrowfill@{\mapstochar\relbar}\relbar\rightarrow}

\renewcommand*\xleftarrow[2][]{\ext@arrow 20{20}0\leftarrowfill@{#1}{#2}}
\providecommand*\xLeftarrow[2][]{\ext@arrow 60{22}0{\Leftarrowfill@}{#1}{#2}}
\providecommand*\xhookleftarrow[2][]{\ext@arrow 10{20}0\hookleftarrowfill@{#1}{#2}}
\providecommand*\xtwoheadleftarrow[2][]{\ext@arrow 60{20}0\twoheadleftarrowfill@{#1}{#2}}
\providecommand*\xleftbararrow[2][]{\ext@arrow 10{22}0\leftbararrowfill@{#1}{#2}}
\providecommand*\xLeftbararrow[2][]{\ext@arrow 50{24}0\Leftbararrowfill@{#1}{#2}}
\providecommand*\xleftringarrow[2][]{\ext@arrow 10{26}0\leftringarrowfill@{#1}{#2}}
\providecommand*\xlefttriarrow[2][]{\ext@arrow 80{24}0\lefttriarrowfill@{#1}{#2}}
\providecommand*\xLefttriarrow[2][]{\ext@arrow 80{24}0\Lefttriarrowfill@{#1}{#2}}

\renewcommand*\xrightarrow[2][]{\ext@arrow 01{20}0\rightarrowfill@{#1}{#2}}
\providecommand*\xRightarrow[2][]{\ext@arrow 04{22}0{\Rightarrowfill@}{#1}{#2}}
\providecommand*\xhookrightarrow[2][]{\ext@arrow 00{20}0\hookrightarrowfill@{#1}{#2}}
\providecommand*\xtwoheadrightarrow[2][]{\ext@arrow 03{20}0\twoheadrightarrowfill@{#1}{#2}}
\providecommand*\xrightbararrow[2][]{\ext@arrow 01{22}0\rightbararrowfill@{#1}{#2}}
\providecommand*\xRightbararrow[2][]{\ext@arrow 04{24}0\Rightbararrowfill@{#1}{#2}}
\providecommand*\xrightringarrow[2][]{\ext@arrow 01{26}0\rightringarrowfill@{#1}{#2}}
\providecommand*\xrighttriarrow[2][]{\ext@arrow 07{24}0\righttriarrowfill@{#1}{#2}}
\providecommand*\xRighttriarrow[2][]{\ext@arrow 07{24}0\Righttriarrowfill@{#1}{#2}}

\providecommand*\xmapsto[2][]{\ext@arrow 01{20}0\mapstofill@{#1}{#2}}
\providecommand*\xleftrightarrow[2][]{\ext@arrow 10{22}0\leftrightarrowfill@{#1}{#2}}
\providecommand*\xLeftrightarrow[2][]{\ext@arrow 10{27}0{\Leftrightarrowfill@}{#1}{#2}}
\providecommand*\xrightarrowtail[2][]{\ext@arrow 0055{\rightarrowtailfill@}{#1}{#2}}

\makeatother

\numberwithin{equation}{section}

\theoremstyle{plain}
\newtheorem{Thm}{Theorem}
\newtheorem*{Thm*}{Theorem}
\newtheorem{Prop}[Thm]{Proposition}
\newtheorem{Cor}[Thm]{Corollary}
\newtheorem{Lemma}[Thm]{Lemma}

\theoremstyle{definition}
\newtheorem{Defn}[Thm]{Definition}

\newtheorem{Ex}[Thm]{Example}

\newtheorem{Rk}[Thm]{Remark}

\newcommand{\twocat}{\mathsf}
\newcommand{\ufa}[2][x]{(\forall_\U #1\,\mathord\in\,#2)\,\!}
\newcommand{\upr}[2][x]{(\Pi_\U #1\,\mathord\in\,#2)\,\!}
\newcommand{\npr}[2][x]{(\Pi #1\,\mathord\in\,#2)\,\!}
\newcommand{\nsig}[2][x]{(\Sigma #1\,\mathord\in\,#2)\,\!}

\newcommand{\uf}{\mathcal{UF}}
\newcommand{\ue}{\mathcal{UE}}
\newcommand{\fc}{\twocat{FC}}
\newcommand{\smc}{\mathop{/_{\!\mathrm{mc}}}}
\newcommand{\cl}[1][\mathbb{T}]{\mathcal{C}l(#1)}

\begin{document}
\title[Ultrafilters and locally connected
  classifying toposes]{Ultrafilters, finite coproducts and locally\\ connected
  classifying toposes}
\author{Richard Garner} 
\address{Centre of Australian Category Theory, Macquarie University, NSW 2109, Australia} 
\email{richard.garner@mq.edu.au}

\subjclass[2000]{Primary: }
\date{\today}

\thanks{The support of Australian Research Council grants
  DP160101519 and FT160100393 is gratefully acknowledged. Thanks also
  to the referees and the editor for helpful and constructive
  suggestions, which have improved the paper significantly.}

\begin{abstract}
  We prove a single category-theoretic result encapsulating the
  notions of ultrafilters, ultrapower, ultraproduct, tensor product
  of ultrafilters, the Rudin--Kiesler partial ordering on
  ultrafilters, and Blass's category of ultrafilters $\uf$. The result
  in its most basic form states that the category
  $\twocat{FC}(\cat{Set}, \cat{Set})$ of finite-coproduct-preserving
  endofunctors of $\cat{Set}$ is equivalent to the presheaf category
  $[\uf, \cat{Set}]$.
  Using this result, and some of its evident generalisations, we
  re-find in a natural manner the important model-theoretic
  \emph{realisation} relation between $n$-types and $n$-tuples of
  model elements; and draw connections with Makkai
  and Lurie's work on conceptual completeness for first-order logic
  via \emph{ultracategories}.
  %


  As a further application of our main result, we use it to describe a
  first-order analogue of J\'onsson and Tarski's \emph{canonical
    extension}. Canonical extension is an algebraic formulation of the
  link between Lindenbaum--Tarski and Kripke semantics for
  intuitionistic and modal logic, and extending it to first-order
  logic has precedent in the \emph{topos of types} construction
  studied by Joyal, Reyes, Makkai, Pitts, Coumans and others. Here, we
  study the closely related, but distinct, construction of the
  \emph{locally connected classifying topos} of a first-order theory.
  The existence of this is known from work of Funk, but the
  description is inexplicit; ours, by contrast, is quite concrete.
\end{abstract}
\maketitle
\leftmargini2em
\section{Introduction}
\emph{Ultrafilters} are important in many areas of mathematics, from
Ramsey theory, to topological dynamics, to universal algebra;
see~\cite{Blass1993Ultrafilters} for an overview. Around the notion of
ultrafilter is a circle of associated concepts: the ultrapower of a
set by an ultrafilter, or more generally, the ultraproduct of a family
of sets~\cite{Frayne1962Reduced}; the tensor product of
ultrafilters~\cite{Katv-etov1968Products} and the more general indexed
sum; and the Rudin--Keisler partial ordering on ultrafilters, first
written down by Blass in~\cite{Blass1970Orderings} and immediately
enhanced to a category of ultrafilters.

Of course, ultrafilters are particularly important in model theory.
One aspect of this is that complete $n$-types of a first-order theory
$\mathbb{T}$ are precisely ultrafilters on the Lindenbaum--Tarski
algebra of $\mathbb{T}$ extended by $n$ new constants; in particular,
each $n$-tuple of elements of a $\mathbb{T}$-model has an associated
complete $n$-type, and this realisation relation between
$n$-tuples and $n$-types is central to many questions in model theory.

Ultrafilters also provide the link between the semantics for modal
logic valued in modal algebras (Boolean algebras with operators), and
the Kripke semantics; indeed, the canonical Kripke model associated to
a modal algebra $B$ takes its frame of possible worlds to comprise
exactly the ultrafilters on $B$. A similar relation holds between the
Lindenbaum--Tarski and Kripke semantics for intuitionistic
propositional logic~\cite[\sec2.6]{Troelstra1988Constructivism}; and
these relations have been expressed algebraically via J\'onsson and
Tarski's notion of \emph{canonical
  extension}~\cite{Jonsson1951Boolean, Gehrke1994Bounded}. Canonical
extension has been generalised from propositional to first-order logic
via the \emph{topos of types} construction studied by Joyal and
Reyes~\cite{Joyal1978Forcing}, Makkai~\cite{Makkai1981The-topos} and
Pitts~\cite{Pitts1983An-application}---and again, ultrafilters play a key role.

Another key use of ultrafilters in model theory is via \L os' theorem
that the models of a first-order theory $\mathbb{T}$ are
closed under ultraproducts. This is useful in its own right, for
example in constructing saturated models, but has also been exploited
more structurally by Makkai~\cite{Makkai1987Stone} and
Lurie~\cite{Lurie2018Ultracategories}: they prove a ``conceptual
completeness'' theorem which can reconstruct a first-order theory
$\mathbb{T}$ (or at least, its completion $\mathbb{T}^{\mathsf{eq}}$
under elimination of imaginaries) from the category of models and
elementary embeddings, together with the ultraproduct structure on
this category.

The objective of this paper is to describe a single category-theoretic
result from which the notions of ultrafilter, ultrapower, tensor
product of ultrafilters, and Blass's category $\uf$ of ultrafilters,
together with their interrelations, all flow naturally---and which,
with only a little more effort, is able to speak towards the
applications of ultrafilters in model theory described above. This
result, which is Theorem~\ref{thm:2} and Corollary~\ref{cor:4} below,
may be stated as follows:

\begin{Thm*}
  \looseness=-1 \leavevmode\!The category $\fc(\cat{Set}, \cat{Set})$
  of finite-coproduct-preserving endofunctors of $\cat{Set}$ is
  equivalent to the category $[\uf, \cat{Set}]$ of functors on Blass'
  category~\cite{Blass1970Orderings} of ultrafilters $\uf$. Under this
  equivalence, the ultrapower functor $(\thg)^\U$ corresponds to the
  representable functor at the ultrafilter~$\U$.
\end{Thm*}



This builds on B\"orger's
characterisation~\cite{Borger1980A-Characterization} of the functor
$\beta \colon \cat{Set} \rightarrow \cat{Set}$, which sends a set $X$
to its set of ultrafilters, as the terminal
finite-coproduct-preserving endofunctor of $\cat{Set}$. (Though it
will not play any role here, we should also
mention~\cite{Kennison1971Equational}'s different characterisation of
$\beta$ as the \emph{codensity monad} of the inclusion functor
$\cat{Fin\S et} \hookrightarrow \cat{Set}$;
see~\cite{Leinster2013Codensity} for a modern account.)

One way of seeing our result is that once we know what a finite
coproduct-preserving endofunctor of $\cat{Set}$ is, everything else is
forced. The ultrapower endofunctors of $\cat{Set}$ arise as the
projective indecomposable objects in $\fc(\cat{Set}, \cat{Set})$, and the full
subcategory they span is equivalent to $\uf^\mathrm{op}$. Moreover, as
we will see in Proposition~\ref{prop:3}, the composition monoidal
structure on $\fc(\cat{Set}, \cat{Set})$ restricts to this
subcategory, and in this way recovers the tensor product of
ultrafilters.

One thing this theorem does not capture is the notion of ultraproduct.
For this, we require a generalisation of the theorem dealing with
ultrafilters not on sets, but on objects of a category $\C$ which is
\emph{extensive}~\cite{Carboni1993Introduction}, meaning that it has
well-behaved finite coproducts. In this context, an ultrafilter on
$X \in \C$ can be defined as an ultrafilter on the Boolean algebra of
coproduct summands of $X$, giving rise to a category $\uf_\C$
generalising Blass' $\uf$. We now obtain the following natural
generalisation of our main theorem, to be proved as
Theorem~\ref{thm:6}:

\begin{Thm*}
  Let $\C$ be extensive. The category $\fc(\C, \cat{Set})$ of
  finite-coproduct-preserving functors from $\C$ to $\cat{Set}$ is
  equivalent to the functor category $[\uf_\C, \cat{Set}]$.
\end{Thm*}

As we will see in Section~\ref{sec:relat-ultraprod}, we may recapture
ultraproducts from this theorem by taking $\C = \cat{Set}^X$, yielding
an equivalence
$[\uf_{\cat{Set}^X}, \cat{Set}] \simeq \fc(\cat{Set}^X, \cat{Set})$;
now the ultraproduct functors
$\Pi_\U \colon \cat{Set}^X \rightarrow \cat{Set}$ correspond under
this equivalence to suitable representable functors in
$[\uf_{\cat{Set}^X}, \cat{Set}]$.

A second application, described in
Section~\ref{sec:relat-model-theory}, takes $\C$ to be the
\emph{classifying Boolean pretopos} of a theory $\mathbb{T}$ of
classical first-order logic, which plays the same role for a
first-order theory as does the Lindenbaum--Tarski algebra of a
propositional theory. In this case, ultrafilters on $A \in \C$
correspond to model-theoretic types in context $A$, and our result
allows us to reconstruct a categorical treatment of
these~\cite{Makkai1981The-topos}. Indeed, by the classifying property
of $\C$, models of $\mathbb{T}$ correspond to pretopos morphisms
$\C \rightarrow \cat{Set}$. As pretopos morphisms preserve finite
coproducts, the theorem thereby associates to each $\mathbb{T}$-model
$M$ a functor $\uf_\C \rightarrow \cat{Set}$---whose values pick out
the sets of $M$-elements that realise each 
type.

Our main theorem can be generalised further by varying the codomain
category as well as the domain category. Recall that a
\emph{Grothendieck topos} is the category of sheaves on a small site,
and that a Grothendieck topos $\E$ is \emph{locally connected} when
the left adjoint $\Delta \colon \cat{Set} \rightarrow \E$ of its
global sections functor
$\Gamma = \E(1, \thg) \colon \E \rightarrow \cat{Set}$ has a further
left adjoint $\pi_0 \colon \E \rightarrow \cat{Set}$. The second
generalisation of our main theorem, to be proved as
Theorem~\ref{thm:8} below, is now:

\begin{Thm*}
  Let $\C$ be extensive and $\E$ a locally connected Grothendieck
  topos. The category $\fc(\C, \E)$ of finite-coproduct-preserving
  functors from $\C$ to $\E$ is equivalent to the functor category
  $[\uf_\C, \E]$.
\end{Thm*}

One application of this theorem, described in
Section~\ref{sec:ultr-ultr-relat}, allows us to reconstruct the
indexed sum of ultrafilters. For any sets $X$ and $Y$, our theorem
yields an equivalence
$\fc(\cat{Set}^X, \cat{Set}^Y) \simeq [\uf_{\cat{Set}^X},
\cat{Set}^{Y}]$, and we define a \emph{generalised ultraproduct}
functor $\cat{Set}^X \rightarrow \cat{Set}^Y$ to be one that
corresponds under this equivalence to a pointwise representable
functor $\uf_{\cat{Set}^X} \rightarrow \cat{Set}^Y$. Such functors
have a representation as \emph{ultraspans}: that is, as diagrams
\begin{equation}\label{eq:28}
  \cd[@-0.5em]{ & M \ar[dl]_-{f} \ar@{~>}[dr]^-{(g, \U)} \\
      X & & Y
    }
\end{equation}
with left leg a function $f$ and right leg a function $g$ endowed with
an ultrafilter $\U_{y}$ on each fibre $g^{-1}y$. Moreover, it turns
out that generalised ultraproduct functors are closed under
composition, so inducing a composition law on ultraspans~\eqref{eq:28}
which encodes perfectly the indexed sum of ultrafilters.

Another potential application of the above theorem, sketched in
Remark~\ref{rk:9}, is to Makkai's
\emph{ultracategories}~\cite{Makkai1987Stone}. An ultracategory is a
category $\C$ endowed with abstract ultraproduct functors
$\Pi_\U \colon \C^X \rightarrow \C$ together with interpretations for
any ``definable map between ultraproducts''---the
\emph{ultramorphisms} of~\cite{Makkai1987Stone}. The key example of an
ultracategory is the category of models of a coherent theory
$\mathbb{T}$ in intuitionistic first-order logic,
and~\cite{Makkai1987Stone}'s main result shows that, to within Morita
equivalence, $\mathbb{T}$ can be reconstructed from its ultracategory
of models.

We expect to relate ultracategories to our main result via the
machinery of \emph{enriched categories}~\cite{Kelly1982Basic,
  Walters1981Sheaves}. We have calculated far enough to convince
ourselves that categories endowed with abstract ultraproduct functors
can be identified with certain categories enriched over the
\emph{bicategory} $\fc_{\cat{Set}}$ of finite-coproduct-preserving
functors between powers of $\cat{Set}$ which admit certain copowers (a
kind of enriched colimit). The key point is that, in proving this, we
exploit the equivalences
$\fc(\cat{Set}^X, \cat{Set}^Y) \simeq [\uf_{\cat{Set}^X},
\cat{Set}^{Y}]$ established above. We will develop this line of
thought further in future work.

Our final main result exploits the preceding theorems to construct the
\emph{locally connected classifying topos} of a suitable pretopos
$\C$. This construction is similar to the \emph{toposes of types}
mentioned above~\cite{Joyal1978Forcing,Makkai1981The-topos,
  Pitts1983An-application}, in that it provides a first-order analogue
to the operation of canonical
  extension~\cite{Jonsson1951Boolean, Gehrke1994Bounded} on
propositional theories.
%
%
%
The existence of locally
connected classifying toposes follows from~\cite{Funk1999The-locally};
however, the existence proof given there is somewhat inexplicit. We
will improve on this by showing that any small pretopos satisfying the
\emph{De Morgan} property (recalled in Definition~\ref{def:25} below)
has a locally connected classifying topos given by the topos of
sheaves on $\uf_\C$ for a certain Grothendieck topology, related to
one found in~\cite{Joyal1978Forcing}. Our final main result, proved as
Theorem~\ref{thm:12} below, is thus:

\begin{Thm*}
  Let $\C$ be a small De Morgan pretopos. The topos
  $\cat{Sh}({\uf_\C})$  is a locally connected classifying topos for $\C$, and is
  itself De Morgan.
\end{Thm*}

While we discuss briefly the similarities and differences between this
construction, and the various toposes of types in the literature, we
will, once again, leave a more detailed comparison to future work.

\section{Background}

\subsection{Ultrafilters, ultraproducts and ultrapowers}
\label{sec:ultrafilters}
In this section, we recall the notions that our main theorem is
designed to capture and their interrelations with each other. Before
starting on this, we first establish some notational conventions for
indexed families which will be used throughout the paper.

\begin{Defn}
  \label{def:12}
  Let $Y = (Y(x) \mid x \in X)$ be an $X$-indexed family of sets. We write
  $\nsig X Y(x)$ or more briefly $X.Y$ for the indexed sum of this
  family, that is, the set of pairs $\{\,(x,y) : x \in X, y \in Y(x)\,\}$. We write
  $\pi_Y \colon X.Y \rightarrow X$ for the first projection map,
  and call this map \emph{the display family associated to $Y$}.
  We also write $\npr X Y(x)$ for the indexed product of the $Y(x)$'s:
  that is, the set of functions $f \colon X \rightarrow X.Y$ which are
  sections of $\pi_Y \colon X.Y \rightarrow X$.
\end{Defn}

More generally, a \emph{display family over $X$} is any function
$\pi \colon E \rightarrow X$, and the \emph{$X$-indexed family
  associated to $\pi$} is the family of fibres
$(\pi^{-1}(x) \mid x \in X)$. As is well known, the passage between
$X$-indexed families and display families over $X$ underlies an
equivalence of categories
\begin{equation}\label{eq:26}
  \cat{Set}^X \simeq \cat{Set}/X\rlap{ .}
\end{equation}
This equivalence and its generalisations will play an important role
in this paper.

\begin{Defn}
  \label{def:1}
  An \emph{ultrafilter} on a set $X$ is a Boolean algebra homomorphism
  $u \colon \P X \rightarrow 2$. Most often, we describe $u$ by
  specifying the subset $\U = u^{-1}(\top)$ of $\P X$; so an
  ultrafilter is equally a collection $\U$ of subsets of $X$
  such that:
  \begin{enumerate}[(i)]
  \item $X \in \U$, and $U \cap V \in \U \Longleftrightarrow (U \in \U \text{ and } V
    \in \U)$; 
  \item[(ii)] $\bot \notin \U$, and
    $U \cup V \in \U \Longleftrightarrow (U \in \U \text{ or } V \in
    \U)$.
  \end{enumerate}
  Equivalently, we may replace condition (ii) with:
  \begin{enumerate}
  \item[(ii)$'$]
    $U \in \U \Longleftrightarrow X \setminus U \notin \U$.
  \end{enumerate}
  We write $\beta X$ for the set of ultrafilters on the set $X$.
\end{Defn}

The \emph{principal ultrafilter} at $x \in X$ is
$\mathord \uparrow x = \{ U \subseteq X : x \in U\}$. These are the
only ultrafilters we can write down explicitly; indeed, the existence
of non-principal ultrafilters is a choice principle, slightly weaker
than the axiom of choice~\cite{Halpern1971The-Boolean}.

It is often useful to view ultrafilters as generalised quantifiers.
Given a predicate $\varphi(x)$ depending on $x \in X$ and an
ultrafilter $\U$ on $X$, we write $\ufa X \varphi(x)$ to indicate that
$\{x \in X : \varphi(x)\} \in \U$ and say that ``for $\U$-almost all
$x$, $\varphi(x)$ holds''.

\begin{Defn}
  \label{def:2}
  Let $\U \in \beta X$. If $Y$ is a set, then the \emph{ultrapower}
  $Y^\U$ is the set of $=_\U$-equivalence classes of partial functions $X \rightharpoonup Y$
  defined on a set in $\U$, where
  \begin{equation}\label{eq:3}
    f =_\U g \qquad \text{iff} \qquad 
    \ufa X f(x) \equiv g(x)\rlap{ .}
  \end{equation}
  Here we write $f(x) \equiv g(x)$ to mean ``$f$ and $g$ are 
  defined at $x$ and are equal''.
  
  More generally, if $Y$ is an $X$-indexed family of sets, then the
  \emph{ultraproduct} $\upr X Y(x)$ is the set of $=_\U$-equivalence
  classes of partial sections, defined on a set in $\U$, of
  $\pi_Y \colon X.Y \rightarrow X$. Note that $Y^\U = \upr X Y$.
\end{Defn}

We can take a topological view on ultraproducts. Given an $X$-indexed
family $Y$, we can view the projection $\nsig X Y(x) \rightarrow X$ as
a local homeomorphism between discrete spaces. Pushing this forward
along the embedding $X \rightarrow \beta X$ of $X$ into its
Stone--C\v ech compactification yields a local homeomorphism over
$\beta X$, whose fibre over an ultrafilter $\U \in \beta X$ is the ultraproduct
$\upr X Y(x)$.

The explicit formula for pushforward of local homeomorphisms---in
terms of germs of local sections---yields the following reformulation of
ultrapowers and ultraproducts in terms of colimits; herein, we view $\U$ as a poset ordered by inclusion:
\begin{equation}
  \label{eq:2}
  \begin{aligned}
    \ Y^\U &= \colim_{U \in \U} Y^U\\
    \upr X Y(x) &= \colim_{U \in \U} \npr U Y(x)\rlap{ .}
  \end{aligned}
\end{equation}  
This description makes it clear that ultraproduct and ultrapower are
functors $(\thg)^\U \colon \cat{Set} \rightarrow \cat{Set}$ and
$\Pi_\U \colon \cat{Set}^X \rightarrow \cat{Set}$ respectively. 

\subsection{The category of ultrafilters}
\label{sec:categ-ultr}
Given ultrafilters $\U$ on $X$ and $\V$ on $Y$, we say that
$f \colon X \rightarrow Y$ is \emph{continuous} if $V \in \V$ implies
$f^{-1}(V) \in \U$. By axiom (ii)$'$ and the fact that $f^{-1}$
preserves complements, this is equally the condition that
\begin{equation}
  \label{eq:1}
  V \in \V \Longleftrightarrow f^{-1}(V) \in \U\rlap{ .}
\end{equation}

The continuous maps play an important role in two natural categories
of ultrafilters, originally defined in~\cite{Katv-etov1968Products,
  Koubek1970On-the-category} in the more general context of
\emph{filters}.

\begin{Defn}
  \label{def:3}
  The category $\ue$ of ultrafilters has pairs $(X \in
  \cat{Set}, \U \in \beta X)$ as objects, and as maps
  $(X, \U) \rightarrow (Y, \V)$ the continuous maps
  $X \rightarrow Y$. The category $\uf$ of ultrafilters has the same
  objects, and as morphisms $(X,\U) \rightarrow (Y, \V)$ the
  $=_\U$-equivalence classes of partial continuous maps
  $X \rightharpoonup Y$ defined on a set in $\U$.
\end{Defn}

Our naming reflects that $\uf$ is the ``good'' category of
ultrafilters and $\ue$ just a preliminary step to get there; for
indeed, $\uf$ arises by inverting the class $\M$ of continuous
injections in $\ue$. The proof of this fact given below mirrors that given
in~\cite[Theorem~16]{Blass1977Two-closed} for the category of filters;
in its statement, $\iota \colon \ue \rightarrow \uf$ is the
identity-on-objects functor taking $f$ to its $=_\U$-equivalence
class.

\begin{Prop}
  \label{prop:2}
  $\iota \colon \ue \rightarrow \uf$ exhibits $\uf$ as $\ue[\M^{-1}]$.
\end{Prop}
\begin{proof}
  Each map in $\M$ factors as an isomorphism followed by a continuous
  subset inclusion; whence $\ue[\M^{-1}] = \ue[\I^{-1}]$ where $\I$ is
  the class of all continuous subset inclusions in $\ue$. It is easy
  to see that any map in $\I$ is of the form
  \begin{equation}\label{eq:6}
    m_{WY} \colon (W, \res \V W) \hookrightarrow (Y, \V)
  \end{equation}
  where $W \in \V$ and $\res \V W = \{U \subseteq W : U \in \V\}$.
  Such maps are stable under composition and contain the
  identities. Moreover, given a map~\eqref{eq:6} and
  $f \colon (X, \U) \rightarrow (Y, \V)$ in $\ue$, we have a
  commuting square (in fact a pullback) in $\ue$ of the form
  \begin{equation*}
    \cd{
      {(f^{-1}(W), \res \U {f^{-1}W})}  \ar[r]^-{} \ar@{
        n->}[d]_{m_{f^{-1}W, X}} &
      {(W, \res \V W)} \ar@{ n->}[d]^{m_{WY}} \\
      {(X, \U)} \ar[r]^-{f} &
      {(Y, \V)}
    }
  \end{equation*}
  since $f^{-1}(W) \in \U$ by continuity of $f$. So $\I$ satisfies the
  first three of the four axioms for a calculus of right
  fractions~\cite{Gabriel1967Calculus}, and satisfies the final one
  trivially since it is a class of monomorphisms. We may thus describe
  the localisation $\ue[\I^{-1}]$ as follows. Objects are those of
  $\ue$, and maps $(X, \U) \rightarrow (Y, \V)$ are spans in $\ue$ as
  to the left below, with two such spans being identified if they can
  be completed to a commuting diagram as to the right.
  \begin{equation*}
    \cd[@-0.5em@C-0.5em]{
      & (U, \res \U U) \ar[dl]_-{m_{UX}} \ar[dr]^-{f} \\
      (X, \U) & & (Y, \V)
    } \qquad \qquad
    \cd[@R-0.5em@-0.5em@C-0.5em]{
      & (U, \res \U U) \ar[dl]_-{m_{UX}} \ar[dr]^-{f} \\
      (X, \U) & (W, \res \U W) 
      \ar[u]|-{m_{WU}} \ar[d]|-{m_{WV}}& (Y, \V) \\
      & (V, \res \U V) \ar[ul]^-{m_{VX}} \ar[ur]_-{g}
    }
  \end{equation*}
  Clearly these maps correspond to
  $=_\U$-equivalence classes of partial continuous functions;
  moreover, under this identification, the identity-on-objects functor
  $\ue \rightarrow \ue[\I^{-1}]$ sends $f$ to $(1,f)$, whence
  $\uf \cong \ue[\I^{-1}]$ under $\ue$ as desired.
\end{proof}

\subsection{Tensor product and indexed sum of ultrafilters}
\label{sec:tens-prod-ultr}

The \emph{tensor product} of ultrafilters is sometimes called the
\emph{product}. It is most easily expressed in terms of generalised
quantifiers.

\begin{Defn}
  \label{def:5}
  Let $\U$ and $\V$ be ultrafilters on $X$ and $Y$. The \emph{tensor
    product} $\U \otimes \V$ is the unique
  ultrafilter on $X \times Y$ which for all
  predicates $\varphi$ on $X \times Y$ satisfies:
  \begin{equation}\label{eq:5}
    (\forall_{\U \otimes \V} (x,y)\,\mathord\in X \times Y)\varphi(x,y)
    \qquad \Longleftrightarrow \qquad 
    (\forall_{\U} x\,\mathord\in X)(\forall_{\V} y\,\mathord\in Y)\varphi(x,y)\rlap{ .}
  \end{equation}
\end{Defn}
Instantiating $\varphi$ at the characteristic predicates of subsets
$A \subseteq X \times Y$ yields the following explicit
formula, wherein we write $x^* A$ for $\{y \in Y : (x,y) \in A\}$:
\begin{equation*}
  \U \otimes \V = \{A \subseteq X \times Y : \{ x \in X : x^* A \in
  \V\} \in \U\}\rlap{ .}
\end{equation*}

Using this formula, we see that if
$f \colon (X, \U) \rightarrow (X', \U')$ and
$g \colon (Y, \V) \rightarrow (Y', \V')$ in $\ue$ then also
$f \times g \colon (X \times Y, \U \otimes \V) \rightarrow (X' \times
Y', \U' \otimes \V')$. So tensor product of ultrafilters gives a
monoidal structure on $\ue$, with as unit the unique ultrafilter on
the one-element set. Since maps in $\M$ are closed under the binary
tensor, this monoidal structure descends along $\iota$ to one on
$\uf$.

The following result, which is a special case
of~\cite[Theorem~1.10]{Frayne1962Reduced}, describes
the interaction of the tensor product with ultrapowers
and ultraproducts.

\begin{Prop}
  \label{prop:4}
  Given $\U \in \beta X$ and $\V \in \beta Y$ and an
  $X \times Y$-indexed family of sets $Z$, currying of functions
  induces an isomorphism of ultraproducts
  \begin{equation}\label{eq:8}
    (\Pi_{\U \otimes \V}
    (x,y)\,\mathord\in\,X \times Y)Z(x,y) \cong (\Pi_\U x\,\mathord\in\,X)(\Pi_\V
    y\,\mathord\in\,Y)Z(x,y)\rlap{ ,}
  \end{equation}
  giving, when $Z$ is a constant family, isomorphisms
  $Z^{\U \otimes \V} \cong (Z^\V)^\U$.
\end{Prop}

A more general construction on ultrafilters is that of \emph{indexed
  sum}.

\begin{Defn}
  \label{def:11}
  Let $\U$ be an ultrafilter on $X$ and, for each $x \in X$,
  let $\V(x)$ be an ultrafilter on $Y(x)$. The \emph{indexed
    sum} $(\Sigma_\U x\! \in\! X)\V(x)$ or $\U.\V$ is the unique
  ultrafilter on $\nsig X Y(x) = X.Y$ which for all predicates
  $\varphi$ on $X.Y$ satisfies
  \begin{equation*}
    (\forall_{\U.\V} (x,y)\,\mathord\in X.Y)\varphi(x,y)
    \quad \Longleftrightarrow \quad 
    (\forall_{\U} x\,\mathord\in X)(\forall_{\V(x)} y\,\mathord\in Y(x))\varphi(x,y)\rlap{ .}
  \end{equation*}
\end{Defn}
Note that when $Y$ and $\V$ are constant families, we have
$(\Sigma_\U x \in X) \V = \U \otimes \V$, so that indexed sum really
does generalise tensor product. Like before, we can obtain an explicit
formula for indexed sum by instantiating at the characteristic
functions of predicates, and like before, we have a formula relating
indexed sums with ultraproducts; this is now the general case
of~\cite[Theorem~1.10]{Frayne1962Reduced}.

\begin{Prop}
  \label{prop:12}
  Given $\U \in \beta X$ and $\V \in \npr X \beta(Y(x))$ and an
  $X.Y$-indexed family of sets $Z$, currying of functions
  induces an isomorphism of ultraproducts
  \begin{equation}\label{eq:24}
    (\Pi_{\U.\V}
    (x,y)\,\mathord\in\,X.Y)Z(x,y) \cong (\Pi_\U x\,\mathord\in\,X)(\Pi_{\V(x)}
    \,y\,\mathord\in\,Y(x))Z(x,y)\rlap{ .}
  \end{equation}
\end{Prop}

\section{The main theorem}
\label{sec:main-theorem}

In this section, we prove our main theorem. This makes essential use
of B\"orger's characterisation~\cite{Borger1980A-Characterization} of the
ultrafilter endofunctor, so we begin by recalling this.

\subsection{B\"orger's theorem}

If $u \colon \P X \rightarrow 2$ is a Boolean algebra homomorphism and
$f \colon X \rightarrow Y$, then
$u \circ (f^{-1}) \colon \P Y \rightarrow \P X \rightarrow 2$ is again
a homomorphism, called the \emph{pushforward} of $u$ along $f$.
Identifying $u$ with the corresponding $\U \subseteq \P X$, its
pushforward along $f$ is given by:
\begin{equation}\label{eq:9}
  f_!(\U) = \{ V \subseteq Y : f^{-1}(V) \in \U\}\rlap{ .}
\end{equation}

\begin{Defn}
  \label{def:6}
  The \emph{ultrafilter endofunctor} $\beta \colon \cat{Set}
  \rightarrow \cat{Set}$ has action on objects $X \mapsto \beta X$ and
  action on morphisms $\beta f \colon \beta X \rightarrow \beta Y$
  given by $\U \mapsto f_!(\U)$.
\end{Defn}

In~\cite{Borger1980A-Characterization} B\"orger characterises $\beta$ as
terminal in the category $\fc(\cat{Set}, \cat{Set})$ of
finite-coproduct-preserving endofunctors of $\cat{Set}$. In
reproducing the proof, and subsequently, we will use the following
lemma, whose proof is either an easy exercise for the reader, or a
consequence of the more general Lemma~\ref{lem:2} below. In the
statement, we call a natural transformation
$\alpha \colon F \Rightarrow G \colon \C \rightarrow \D$
\emph{monocartesian} if the naturality square of $\alpha$ at any
monomorphism $f \colon X \rightarrowtail Y$ is a
pullback.

\begin{Lemma}
  \label{lem:8}
  Let $G \colon \cat{Set} \rightarrow \cat{Set}$ preserve finite
  coproducts and let $\alpha \colon F \Rightarrow G$.
  \begin{enumerate}[(i)]
  \item $G$ preserves monomorphisms and pullbacks along monomorphisms;
  \item $F$ preserves finite coproducts if and only if
    $\alpha$ is monocartesian.
  \end{enumerate}
\end{Lemma}

\begin{Thm}\cite[Theorem~2.1]{Borger1980A-Characterization}
  \label{thm:3}
  $\beta$ is terminal in $\fc(\cat{Set}, \cat{Set})$.
\end{Thm}
\begin{proof}
  For an injection $f \colon X \rightarrow Y$, the map
  $\beta f \colon \beta X \rightarrow \beta Y$ is also injective with
  \begin{equation}\label{eq:16}
    \im \beta f = \{\V \in \beta Y : f(X) \in \V \}\rlap{ .}
  \end{equation}
  Indeed, since $f$ is injective,
  $f^{-1} \colon \P Y \rightarrow \P X$ is surjective and so
  $\beta f = (\thg) \circ (f^{-1})$ is injective. As for its image:
  each $\U \in \beta X$ contains $X = f^{-1}(f(X))$, so
  by~\eqref{eq:9} each $f_!(\U)$ contains $f(X)$. Conversely, if
  $\V \in \beta Y$ contains $f(X)$, then
  $\U = \{U \subseteq X: f(U) \in \V\}$ is an ultrafilter on $X$ with
  $f_!(\U) = \V$.

  We first use this to show $\beta \in \fc(\cat{Set}, \cat{Set})$.
  Clearly $\beta(\emptyset) = \emptyset$; while if we have a coproduct
  ${y_1 \colon Y_1 \rightarrow Y \leftarrow Y_2 \colon y_2}$, then the
  maps
  $\beta y_1 \colon \beta Y_1 \rightarrow \beta Y \leftarrow \beta Y_2
  \colon \beta y_2$ are each injective with as images the sets
  $A = \{\,\U \in \beta Y : \im y_1 \in \U\,\}$ and
  $B = \{\,\U \in \beta Y : \im y_2 \in \U\,\}$. Since $\im y_1$ and
  $\im y_2$ partition $Y$, each $\U \in \beta Y$ lies in
  \emph{exactly} one of $A$ or $B$ whence $(\beta y_1, \beta y_2)$ is
  again a coproduct cone.

  We now show $\beta$ is terminal in $\fc(\cat{Set}, \cat{Set})$.
  Given $T \in \fc(\cat{Set}, \cat{Set})$ and
  $x \in TX$, define the \emph{type} of $x$ as the ultrafilter
  on $X$ given by:
  \begin{equation*}
    \tau_X(x) = \{U \subseteq X : \text{$x$ is in the image of the monic $TU \rightarrowtail
      TX$}\}\rlap{ .}
  \end{equation*}
  Here, $TU \rightarrowtail TX$ is the $T$-image of the inclusion
  $U \subseteq X$, and so monic by Lemma~\ref{lem:8}. 
  $\tau_X(x)$ satisfies axiom (i) for an ultrafilter since
  $T$ preserves pullbacks of monics, and satisfies (ii)$'$ as
  $TU \rightarrowtail TX \leftarrowtail T(X \setminus U)$ is the
  $T$-image of a coproduct diagram and so itself a coproduct.

  So we have functions $\tau_X \colon TX \rightarrow \beta X$. To
  verify their naturality in $X$ we must show for any $x \in TX$ and
  $f \colon X \rightarrow Y$ that $\tau_Y(Tf(x)) = f_!(\tau_X(x))$. So
  for any $V \subseteq Y$, we must show $Tf(x) \in TY$ is in the image
  of
  $TV \rightarrowtail TY$ if and only if $x \in TX$ is in the image of
  $T(f^{-1}(V)) \rightarrowtail TX$; which is so because $T$ preserves
  the pullback of $V \rightarrowtail Y$ along
  $f \colon X \rightarrow Y$ by Lemma~\ref{lem:8}. So we have
  $\tau \colon T \Rightarrow \beta$.

  Finally, we check uniqueness of $\tau$. Any
  $\sigma \colon T \Rightarrow \beta$ is monocartesian
  by Lemma~\ref{lem:8}; and so for each $m \colon U \subseteq X$ the
  following square is a pullback:
  \begin{equation*}
    \cd[@-0.7em]{
      {TU} \ar@{ >->}[r]^-{Tm} \ar[d]_{\sigma_V} &
      {TX} \ar[d]^{\sigma_X} \\
      {\beta U} \ar@{ >->}[r]^-{\beta m} &
      {\beta X}\rlap{ .}
    }
  \end{equation*}
  Thus, $x \in TX$ factors through $Tm$ if and only if $\sigma_X(x)$
  factors through $\beta m$ which by~\eqref{eq:16} happens just when
  $U \in \sigma_X(x)$. So $\sigma_X(x) = \tau_X(x)$ as desired.
\end{proof}

\subsection{The main theorem}
\label{sec:proof-main-theorem}

We now exploit B\"orger's theorem to prove our main
Theorem~\ref{thm:2}. In doing so, we make use of the well-known
generalisation of~\eqref{eq:26} stating that~\emph{any slice of a
  presheaf category is equivalent to a presheaf category}.

Indeed, given $X \in [\A, \cat{Set}]$, the \emph{category of elements}
$\el X$ has as objects, pairs $(A \in \A, x \in XA)$ and as morphisms
$(A,x) \rightarrow (A',x')$, maps $f \colon A \rightarrow A'$ in $\A$
such that $x' = Xf(x)$. The equivalence in question is now
\begin{equation}\label{eq:10}
  [\el X,\cat{Set}] \simeq [\A, \cat{Set}] / X \rlap{ ,}
\end{equation}
and is constructed by applying~\eqref{eq:26} componentwise as follows.
Going from left to right, $Y \colon \el X \rightarrow \cat{Set}$ is
sent to $\pi \colon \int\!Y \rightarrow X$ whose $A$-component is
given by the first projection map $\nsig {XA} Y(A,x) \rightarrow XA$,
and where the action of $\int\!Y$ on maps is induced from those of $X$
and $Y$. Going from right to left, $p \colon E \rightarrow X$ in
$[\A, \cat{Set}]/X$ is sent to $\tilde E$ in $[\el X, \cat{Set}]$ with
$\smash{\tilde E(A,x) = p_A^{-1}(x)} \subseteq EA$ and action on maps
inherited from $E$. For a detailed proof of the equivalence, see for
example~\cite[Proposition~A1.1.7]{Johnstone2002Sketches}.

\begin{Thm}
  \label{thm:2}
  The category $\fc(\cat{Set}, \cat{Set})$ of finite
  coproduct-preserving endofunctors of $\cat{Set}$ is equivalent to
  $[\uf, \cat{Set}]$. 
\end{Thm}

\begin{proof}
  Note that $T \in [\cat{Set}, \cat{Set}]$ preserves finite coproducts
  if and only if it admits a monocartesian transformation to $\beta$,
  which is then necessarily unique. The ``if'' direction of this claim
  follows from Lemma~\ref{lem:8}; whereupon the ``only if'' direction
  and the unicity follow from Theorem~\ref{thm:3}. So we have an
  isomorphism of categories
  \begin{equation}\label{eq:12}
   \fc(\cat{Set}, \cat{Set}) \cong [\cat{Set}, \cat{Set}] \smc \beta
  \end{equation}
  where to the right we have the full, replete, subcategory
  $[\cat{Set}, \cat{Set}] \smc \beta$ of
  $[\cat{Set}, \cat{Set}] / \beta$ on the monocartesian arrows.

  Now, the full slice category $[\cat{Set}, \cat{Set}] / \beta$ is
  equivalent to $[\el \beta, \cat{Set}]$. Here, objects of
  $\el \beta$ are pairs $(X \in \cat{Set},\, \U \in \beta X)$, while maps
  $(X, \U) \rightarrow (Y, \V)$ are functions
  $f \colon X \rightarrow Y$ such that $f_!(\U) = \V$.
  Comparing~\eqref{eq:1} with~\eqref{eq:9}, these are exactly the
  \emph{continuous} maps, so that $\el \beta \cong \ue$
  and~\eqref{eq:10} becomes an equivalence:
  \begin{equation}\label{eq:11}
    [\cat{Set}, \cat{Set}] / \beta \simeq [\ue, \cat{Set}]\rlap{ .}
  \end{equation}

  An object $\tau \colon T \Rightarrow \beta$ to the left of this
  equivalence lies in the full replete subcategory
  $[\cat{Set}, \cat{Set}] \smc \beta$ just when for each monic
  $f \colon X \rightarrowtail Y$ and $\U \in \beta X$, the map on
  fibres $\tau_X^{-1}(\U) \rightarrow \tau_Y^{-1}(f_!(\U))$ is an
  isomorphism. This is equally the condition that the corresponding
  $\bar \tau \in [\ue, \cat{Set}]$ to the right lies in the full,
  replete subcategory of functors which send the class $\M$ of
  continuous injective functions to isomorphisms. By
  Proposition~\ref{prop:2}, this subcategory is isomorphic to
  $[\uf, \cat{Set}]$ via restriction along
  $\iota \colon \ue \rightarrow \uf$. So~\eqref{eq:11} restricts to an
  equivalence
  $[\cat{Set}, \cat{Set}] \smc \beta \simeq [\uf, \cat{Set}]$, and
  combining this with~\eqref{eq:12} yields the desired equivalence
  $\fc(\cat{Set}, \cat{Set}) \simeq [\uf, \cat{Set}]$.
\end{proof}

Chasing through the above equivalences, we see that for each $A \in
\fc(\cat{Set}, \cat{Set})$, the corresponding $\tilde A \colon \uf
\rightarrow \cat{Set}$ is defined on objects by
\begin{equation}\label{eq:17}
  \tilde A(X, \U) = \{\,x \in AX : \tau_X(x) = \U\,\} \cong \textstyle\bigcap_{U \in
    \U} \im AU \subseteq AX
  \rlap{ .}
\end{equation}
For its definition on morphisms, let the map
$(X, \U) \rightarrow (Y, \V)$ of $\uf$ be represented by the partial
continuous $f \colon X \rightharpoonup Y$ defined on $U \in \U$. Then
the induced function $\tilde A(X, \U) \rightarrow \tilde A(Y,\V)$ is
defined by $x \mapsto Af(x')$, where $x' \in AU$ is the lifting of $x$
through $AU \rightarrowtail AX$ guaranteed by the fact that
$U \in \tau_X(x)$.

In the other direction, for any $B \colon \uf \rightarrow \cat{Set}$,
the corresponding finite-coproduct-preserving
$\int\!B \colon \cat{Set} \rightarrow \cat{Set}$ is defined by
\begin{equation}\label{eq:13}
  (\textstyle\int\!B)Y = \sum_{\V\in\beta Y} B(Y, \V) \quad \text{and} \quad
  (\int\!B)f \colon (\V, a) \mapsto (f_!(\V), Af(a))\rlap{ .}
\end{equation}

\subsection{Relation to ultrapowers and tensor products}
\label{sec:relation-ultrapowers}
We now show that both ultrapowers and the tensor product of
ultrafilters arise naturally from the preceding
equivalence. We begin with ultrapowers. 

\begin{Cor}
  \label{cor:4}
  Under the equivalence of Theorem~\ref{thm:2}, the representable
  functor $\uf( (X, \U), \thg) \colon \uf \rightarrow \cat{Set}$
  corresponds to the ultrapower functor
  $(\thg)^\U \colon \cat{Set} \rightarrow \cat{Set}$.
\end{Cor}

\begin{proof}
  Taking $B$ to be 
  $y_{(X,\U)} = \uf((X, \U), \thg)$ in~\eqref{eq:13}, we have that
  \begin{equation*}
    (\textstyle\int\!y_{(X,\U)})(X) = \textstyle\sum_{\V \in \beta Y} \uf( (X, \U), (Y, \V) ) \rlap{ .}
  \end{equation*}
  An element of this set is a pair $\V \in \beta Y$ together with an
  $=_\U$-equivalence class of continuous partial functions
  $f \colon (X,\U) \rightharpoonup (Y,\V)$ defined on a set in $\U$.
  The continuity condition~\eqref{eq:1} forces $\V = f_!(\U)$ and so
  this is equally a $=_\U$-equivalence class of partial functions
  $f \colon X \rightharpoonup Y$ defined on a set in $\U$; thus an
  element of the ultrapower $Y^\U$. This proves that
  $\textstyle\int\!y_{(X,\U)} \cong (\thg)^\U$ as desired.
\end{proof}

As remarked in the introduction, we can use this result to recover the
category $\uf$ from $\fc(\cat{Set}, \cat{Set})$. Recall that an
object $X$ of a category $\E$ is \emph{small-projective} if the
hom-functor $\E(X, \thg) \colon \E \rightarrow \cat{Set}$ preserves
all small colimits.

\begin{Cor}
  \label{cor:2}
  The category $\uf^\mathrm{op}$ is equivalent to the full subcategory of
  $\fc(\cat{Set}, \cat{Set})$ on the small-projectives.
\end{Cor}
\begin{proof}
  For any locally small $\A$, the small-projectives in
  $[\A, \cat{Set}]$ are precisely the retracts of representable
  functors; see, for
  example,~\cite[Lemma~6.5.10]{Borceux1994Handbook1}. So if
  representables are closed under retracts, then $\A^\mathrm{op}$ is
  equivalent to the full subcategory of $[\A, \cat{Set}]$ on the
  small-projectives.

  It thus suffices to show that
  representables in $[\uf, \cat{Set}]$ are closed under retracts. But
  if $i \colon A \rightarrowtail y_{(X, \U)}$ and
  $p \colon y_{(X,\U)} \twoheadrightarrow A$ with $pi=1$, then
  $ip \colon y_{(X,\U)} \rightarrow y_{(X,\U)}$ is the image under $y$
  of an idempotent on $(X,\U)$. Since
  by~\cite[Theorem~5]{Blass1970Orderings}, the only idempotents
  (indeed, the only endomorphisms) in $\uf$ are the identities, we
  thus have $ip = 1$ as well as $pi = 1$, so that $A \cong y_{(X,\U)}$ is
  again representable.
\end{proof}

We now turn to the tensor product of ultrafilters. The category
$\fc(\cat{Set}, \cat{Set})$ has a monoidal structure given by
composition, and transporting this across the equivalence of
Theorem~\ref{thm:2} yields a monoidal structure $(I, \otimes)$ on
$[\uf, \cat{Set}]$.

\begin{Prop}
  \label{prop:3}
  The representables in $[\uf, \cat{Set}]$ are closed under the
  monoidal structure, and the induced monoidal structure on $\uf$
  is that given by tensor product of ultrafilters.
\end{Prop}
\begin{proof}
  The identity functor $\cat{Set} \rightarrow \cat{Set}$ corresponds
  to the functor $\uf \rightarrow \cat{Set}$ represented by the unique
  ultrafilter on a one-element set, which is the unit for the
  monoidal structure on $\uf$. On the other hand, if
  $A,B \in [\uf, \cat{Set}]$ are represented by $(X, \U)$ and $(Y,\V)$
  respectively, then by Theorem~\ref{thm:2} we have
  $\textstyle\int\!A \cong (\thg)^\U$ and
  $\textstyle\int\!B \cong (\thg)^\V$, and so
  $\textstyle\int\!A \circ \textstyle\int\!B \cong ((\thg)^\V)^\U
  \cong (\thg)^{\U \otimes \V}$ by Proposition~\ref{prop:4}. It
  follows that $A \otimes B$ is represented by
  $(X \times Y, \U \otimes \V) = (X, \U) \otimes (Y, \V)$ in $\uf$.
\end{proof}

The monoidal structure on $[\uf, \cat{Set}]$ is easy to write down
explicitly. As already noted, the unit $I$ is the functor
representable at the unique ultrafilter on the one-element set, while
the binary tensor can be given as $A \otimes B = \int\!A \circ B$,
where $\int\! A$ is defined as in~\eqref{eq:13}. This yields the
formulae:
\begin{align*}
  I(X, \U) &= \begin{cases} 1 & \text{if $\U$ is principal;} \\ 0 &
    \text{otherwise.}
  \end{cases}\\
  (A \otimes B)(X, \U) &= \textstyle \sum_{\V \in \beta(B(X, \U))}
  A(B(X, \U), \V)\rlap{ .}
\end{align*}

We have an alternative description of the binary tensor product by
exploiting the fact that,
since the composition product on $\fc(\cat{Set}, \cat{Set})$
preserves colimits in its first variable, so too does the tensor
product $\otimes$ on $[\uf, \cat{Set}]$:
\begin{equation}
  \label{eq:31}
  \begin{aligned}
    (A \otimes B)(X,\U) &\cong \left((\textstyle\int^{(Y, \V) \in \uf} A(Y, \V)
      \times y_{(Y,\V)})
      \otimes B\right)(X,\U) \\ &\cong \textstyle\int^{(Y, \V) \in \uf} A(Y, \V) \times (y_{(Y,\V)}
    \otimes B)(X,\U)\\
    &\cong \textstyle\int^{(Y, \V) \in \uf} A(Y, \V) \times B(X,\U)^\V\rlap{ .}
  \end{aligned}
\end{equation}

Compare this with the well-known \emph{substitution monoidal structure} on
$[\cat{F}, \cat{Set}]$---for $\cat{F}$ the category of functions
between finite cardinals---defined by:
\begin{equation}\label{eq:35}
  (A \otimes B)(m) = \textstyle\int^{n \in \cat{F}} An \times Bm^n\rlap{ .}
\end{equation}

\begin{Rk}
  \label{rk:4}
Given a monoidal category $\V$, one may consider categories
\emph{enriched} over $\V$ in the sense of~\cite{Kelly1982Basic}. A
$\V$-enriched category $\C$ involves a set of objects $A,B,C,\dots$ as usual, but
instead of hom-sets of morphisms, one has hom-\emph{objects} $\C(A,B)$
in $\V$. In the enriched context, one has a new kind of colimit
available, namely the \emph{copower} $V \cdot A$ of an object $A \in \C$ by an object $V \in \V$, characterised by natural isomorphisms $\C(V \cdot A, B) \cong [V, \C(A,B)]$ in $\V$.

In~\cite{Garner2014Lawvere}, the author considered categories enriched
over $[\F, \cat{Set}]$ with the substitution monoidal structure, and
showed that such $[\F, \cat{Set}]$-categories admitting copowers by
representables correspond to ordinary categories $\C$ admitting finite
powers $(\thg)^n \colon \C \rightarrow \C$. The analogy
between~\eqref{eq:31} and~\eqref{eq:35} suggests that something
similar should be possible with $[\uf, \cat{Set}]$ in place of
$[\F, \cat{Set}]$, and this is indeed so: categories enriched over
$[\uf, \cat{Set}]$ with copowers by representables correspond to
ordinary categories $\C$ equipped with \emph{abstract ultrapower}
functors $(\thg)^\U \colon \C \rightarrow \C$. The details of this
will be left for future work, but we discuss an extension to
categories endowed with \emph{abstract ultraproduct} functors in
Remark~\ref{rk:9} below.
\end{Rk}

\section{First generalisation}
\label{sec:gener-main-theor}

In this section, we give our first generalisation of
Theorem~\ref{thm:2}. This will show that, for any extensive category
$\C$, there is an equivalence
$\fc(\C, \cat{Set}) \simeq [\uf_\C, \cat{Set}]$ where $\uf_\C$ is a
suitably defined category of ultrafilters on $\C$-objects.
We then describe how this result captures the notion of ultraproduct,
and how it reconstructs the categorical treatment
in~\cite{Makkai1981The-topos} of types in model theory.

\subsection{Generalising the domain category}
\label{sec:gener-doma-categ}

We begin by recalling from~\cite{Carboni1993Introduction} that a
category $\C$ with finite coproducts is \emph{extensive} if for every
$A, B \in \C$, the functor
$ + \colon \C / A \times \C / B \rightarrow \C / (A+B)$ is an
equivalence of categories. Equally,
by~\cite[Proposition~2.2]{Carboni1993Introduction}, $\C$ is extensive
just when it has pullbacks along coproduct coprojections, and, for
every diagram
\begin{equation}\label{eq:14}
  \cd[@-0.5em]{
    A' \ar[d]_-{a} \ar[r]^-{i'} & C' \ar[d]_-{c} & \ar[l]_-{j'} B'
    \ar[d]_-{b} \\
    A \ar[r]^-{i} & C & B \ar[l]_-{j}
  }
\end{equation}
in which the bottom row is a coproduct diagram, the top row is a
coproduct diagram if and only if the two squares are pullbacks.

Note the ``if'' direction says that binary coproducts in $\C$ are
pullback-stable. In fact, a category with finite coproducts and
pullbacks along their coprojections is extensive just when binary
coproducts are pullback-stable and \emph{disjoint};
see~\cite[Proposition~2.14]{Carboni1993Introduction}. Disjointness
means that coproduct coprojections are monic, and the pullback of the
two coprojections of a binary coproduct is initial.

This characterisation implies that any topos is extensive;
see~\cite[Proposition~A2.3.4 \&
Corollary~A2.4.4]{Johnstone2002Sketches}. In particular, $\cat{Set}$
is extensive, as is any presheaf category. Other examples of extensive
categories include the categories of topological spaces, of small
categories and of affine schemes.

Let us write $\mathrm{Sum}_\C(X)$ for the poset of coproduct summands
of $X \in \C$: that is, the poset of isomorphism-classes of coproduct
coprojections with codomain $X$. 

\begin{Prop}
  \label{prop:5}
  If $\C$ is extensive, then for each $X \in \C$ the poset
  $\mathrm{Sum}_\C(X)$ is a Boolean algebra, and for each
  $f \colon X \rightarrow Y$, pullback along $f$ defines a Boolean
  algebra homomorphism
  $f^{-1} \colon \mathrm{Sum}_\C(Y) \rightarrow \mathrm{Sum}_\C(X)$.
\end{Prop}

\begin{proof}
  Since binary coproducts in $\C$ are stable under pullback, so are
  coproduct coprojections; since they are also composition-closed,
  each $\mathrm{Sum}_\C(X)$ has finite meets, and
  $f^{-1} \colon \mathrm{Sum}_\C(Y) \rightarrow \mathrm{Sum}_\C(X)$ is
  well-defined and finite-meet-preserving.
  Now any $Y_1 \rightarrowtail Y$ in $\mathrm{Sum}_\C(X)$ is part of a
  coproduct $Y_1 \rightarrowtail Y \leftarrowtail Y_2$. Of course
  $Y_1 \cup Y_2 = \top_Y$ in $\mathrm{Sum}_\C(Y)$, and
  $Y_1 \cap Y_2 = \bot_Y$ by disjointness; so $\mathrm{Sum}_\C(X)$ has
  complements and so is a Boolean algebra. Further,
  $f^{-1} \colon \mathrm{Sum}_\C(Y) \rightarrow \mathrm{Sum}_\C(X)$
  preserves these complements as binary coproducts are
  pullback-stable.
\end{proof}

A case worth noting is that where $\C$ is \emph{Boolean extensive},
meaning that every monic in $\C$ is a coproduct coprojection; in this
situation $\mathrm{Sum}_\C(X)$ coincides with the full subobject
lattice $\mathrm{Sub}_\C(X)$, so that all subobject lattices in $\C$
are Boolean algebras---whence the nomenclature. In particular, the
category of sets is Boolean extensive, so that the following result is
a generalisation of Lemma~\ref{lem:8} above. In the last part of the
statement, a natural transformation $\alpha$ is called
\emph{sum-cartesian} if its naturality square at every coproduct
coprojection is a pullback.

\begin{Lemma}
  \label{lem:2}
  Let $\C$ and $\D$ be extensive, let $G \colon \C \rightarrow \D$ be
  finite-coproduct-preserving and let $\alpha \colon F \Rightarrow G
  \colon \C \rightarrow \D$.
  \begin{enumerate}[(i)]
  \item $G$ preserves both coproduct coprojections and pullbacks along
    such;
  \item $F$ preserves finite coproducts just when $\alpha \colon F
    \Rightarrow G$ is sum-cartesian.
  \end{enumerate}
\end{Lemma}

\begin{proof}
  The first part of (i) is clear. For the second, any pullback
  along a coproduct coprojection in $\C$ is the left square of a
  diagram like~\eqref{eq:14} in which both rows are coproducts.
  Applying $F$, both rows remain coproducts and so by extensivity of
  $\D$, both squares remain pullbacks. As for (ii), given a coproduct diagram
  $i \colon A \rightarrow C \leftarrow B \colon j$ in $\C$, we
  consider the diagram
  \begin{equation*}
    \cd[@-0.5em]{
      FA \ar[d]_-{\alpha_A} \ar[r]^-{Fi} & FC \ar[d]_-{\alpha_C} & \ar[l]_-{Fj} FB
      \ar[d]_-{\alpha_B} \\
      GA \ar[r]^-{Gi} & GC & GB \ar[l]_-{Gj}
    }
  \end{equation*}
  in $\D$. The bottom row is a coproduct since $G$ preserves such;
  so, by extensivity of $\D$, the top row is a coproduct (i.e.,
  $F$ preserves finite coproducts) just when both squares are
  pullbacks (i.e., $\alpha$ is sum-cartesian).
\end{proof}

If $\C$ is extensive, then we define an \emph{ultrafilter} on
$X \in \C$ to be a Boolean algebra homomorphism
$\mathrm{Sum}_\C(X) \rightarrow 2$; equivalently, a subset
$\U \subseteq \mathrm{Sum}_\C(X)$ satisfying the analogue of
conditions (i) and (ii) or (i) and (ii)$'$ of Definition~\ref{def:1}.
Like before, we write $\beta X$ for the set of ultrafilters on
$X \in \C$. Since each
$f^{-1} \colon \mathrm{Sum}_\C(Y) \rightarrow \mathrm{Sum}_\C(X)$ is a
Boolean algebra homomorphism, precomposition with $f^{-1}$ yields a
function $\beta f \colon \beta X \rightarrow \beta Y$; in this way, we
define an ultrafilter functor $\beta \colon \C \rightarrow \cat{Set}$.

\begin{Prop}
  \label{prop:1}
  $\beta \colon \C \rightarrow \cat{Set}$ is terminal in
  $\fc(\C, \cat{Set})$.
\end{Prop}
\begin{proof}
  The proof in Theorem~\ref{thm:3} adapts without difficulty to show
  that $\beta \colon \C \rightarrow \cat{Set}$ preserves
  finite coproducts. To show terminality in
  $\fc(\C, \cat{Set})$, we suppose given
  $T \in \fc(\C, \cat{Set})$. For any $X \in \C$ and 
  $x \in TX$, we again define the \emph{type} of $x$ to be:
  \begin{equation*}
    \tau_X(x) = \{ U \xrightarrowtail{m} X \in \mathrm{Sum}_\C(X)
    : \text{$x$ factors through } FU \xrightarrowtail{Fm} FX\}\rlap{ .}
  \end{equation*}
  The same argument as before, but now exploiting Lemma~\ref{lem:2} in
  place of Lemma~\ref{lem:8}, shows that this definition gives
  the values of a well-defined natural transformation 
  $\tau \colon T \Rightarrow \beta \colon \C \rightarrow \cat{Set}$,
  and that this $\tau$ is unique.
\end{proof}

Like before, given objects $X,Y \in \C$ endowed with ultrafilters $\U$
and $\V$, we call $f \colon X \rightarrow Y$ \emph{continuous}
if for all $V \in \mathrm{Sum}_\C(Y)$ we have
$V \in \V \Leftrightarrow f^{-1}(V) \in \U$. More generally, a map
$f \colon U \rightarrow Y$ defined on the domain of some
$U \rightarrowtail X$ in $\U$ is \emph{continuous} if
$V \rightarrowtail Y \in \V$ just when
$f^{-1}(V) \rightarrowtail U \rightarrowtail X \in \U$. Two
partial maps defined on $U$ and $U'$ are \emph{$=_\U$-equivalent} if
their restrictions to some $W \subseteq U \cap U'$ in $\U$ coincide.

\begin{Defn}
  \label{def:7}
  The category $\ue_\C$ has pairs $(X \in \C, \U \in \beta X)$ as
  objects, and as morphisms $(X, \U) \rightarrow (Y, \V)$ the
  continuous maps $X \rightarrow Y$. The category $\uf_\C$ has the
  same objects, and as morphisms $(X,\U) \rightarrow (Y, \V)$ the
  $=_\U$-equivalence classes of partial continuous maps
  $X \rightharpoonup Y$ defined on the domain of some
  $U \rightarrowtail X$ in $\U$.
\end{Defn}
Writing $\M_\Sigma$ for the maps in $\ue_\C$ whose underlying map in
$\C$ is a coproduct coprojection, we have as in
Proposition~\ref{prop:2}, that $\uf_\C \cong \ue_\C[\M_\Sigma^{-1}]$.
Now transcribing the proof of Theorem~\ref{thm:2} and
Corollary~\ref{cor:4}, but exploiting Proposition~\ref{prop:1} and
Lemma~\ref{lem:2} in place of Theorem~\ref{thm:3} and
Lemma~\ref{lem:8}, gives the following.

\begin{Thm}
  \label{thm:6}
  Let $\C$ be extensive. The category $\fc(\C, \cat{Set})$ of finite
  coproduct-preserving functors $\C \rightarrow \cat{Set}$ is
  equivalent to $[\uf_\C, \cat{Set}]$. Under this equivalence, the
  representable presheaf at $(X,\U) \in \uf_\C$ corresponds to the
  ``ultrahom functor''
  \begin{equation*}
    \C(X,\thg)_\U = \mathrm{colim}_{U \in \U} \C(U, \thg) \colon \C
    \rightarrow \cat{Set}\rlap{ .}
  \end{equation*}
\end{Thm}
The formulae for the two directions of the equivalence
$\fc(\C, \cat{Set}) \simeq [\uf_\C, \cat{Set}]$ are once again given
by~\eqref{eq:17} and~\eqref{eq:13}. 

\begin{Ex}
  \label{ex:2}
  The category $\cat{Stone}$ of Stone spaces is extensive and
  $\mathrm{Sum}_{\cat{Stone}}(X)$ is the Boolean algebra of clopen
  sets of $X$. It follows by Stone duality that ultrafilters on
  $X \in \cat{Stone}$ correspond exactly to points of $X$, so that
  the category $\uf_{\cat{Stone}}$ has pointed Stone spaces $(X,x)$ as
  objects, and as maps $f \colon (X,x) \rightarrow (Y,y)$, germs at
  $x$ of point-preserving continuous functions $X \rightarrow Y$.
  Under the equivalence
  $[\uf_\cat{Stone}, \cat{Set}] \simeq \twocat{FC}(\cat{Stone},
  \cat{Set})$, the representable at $(X,x)$ corresponds to the functor
  which sends a Stone space $Y$ to the stalk at $x$ of the sheaf of
  continuous functions $X \rightarrow Y$.
\end{Ex}

\subsection{Relation to ultraproducts}
\label{sec:relat-ultraprod}
We now explain how Theorem~\ref{thm:6} allows us to reconstruct the
notion of ultraproduct. Taking $\C = \cat{Set}^X$ therein yields the
equivalence
$[\uf_{\cat{Set}^X}, \cat{Set}] \simeq \fc(\cat{Set}^X, \cat{Set})$,
and we will obtain the ultraproduct functors as correlates to the
right of suitable \emph{representable} functors to the left.

Note first that,
via~\eqref{eq:26}, we have for any $A \in \cat{Set}^X$ that
\begin{equation*}
  \mathrm{Sum}_{\cat{Set}^X}(A) \cong \mathrm{Sum}_{\cat{Set}/X}(\pi_A
  \colon X.A \rightarrow X) \cong \P(X.A)\rlap{ ,}
\end{equation*}
so that ultrafilters on $A \in \cat{Set}^X$ can be identified with
ultrafilters on $X.A$ in $\cat{Set}$. Under this identification, the
ultrafilter $\U$ on $X.A$ corresponds to the ultrafilter
$\tilde \U = \{\tilde U : U \in \U\}$ on $A$ composed of the subobjects
$\tilde U \rightarrowtail A$ obtained by passing the subobjects
$U \rightarrowtail X.A \rightarrow X$ of
$\pi_A \colon X.A \rightarrow X$ across the equivalence~\eqref{eq:26}.

\begin{Prop}
  \label{prop:6}
  Under the equivalence
  $[\uf_{\cat{Set}^X}, \cat{Set}] \simeq \fc(\cat{Set}^X, \cat{Set})$,
  the representable functor at
  $(A \in \cat{Set}^X, \U \in \beta(X.A))$ corresponds to the composite
  \begin{equation}\label{eq:33}
    \cat{Set}^X \xrightarrow{\ \ \textstyle\cat{Set}^{\pi_A}\ \ } \cat{Set}^{X.A}
    \xrightarrow{\ \ \textstyle\Pi_\U\ \ } \cat{Set}\rlap{ .}
  \end{equation}
  In particular, the representable at $(1, \U)$ corresponds to
  $\Pi_\U \colon \cat{Set}^X \rightarrow \cat{Set}$.
\end{Prop}

\begin{proof}
  From Theorem~\ref{thm:6} and the above remarks, we know that
  $y_{(A, \U)}$ corresponds to the ultrahom functor $\cat{Set}^X(A,
  \thg)_{\tilde \U}$. We now calculate that:
  \begin{align*}
    \cat{Set}^X(A, Y)_{\tilde \U} &=  \mathop{\mathrm{colim}}\nolimits_{\tilde U \in \tilde
      \U}
    \cat{Set}^X(\smash{\tilde U}, Y) \\
    &\cong  \mathop{\mathrm{colim}}\nolimits_{U \in \U} \cat{Set}/X(\,U
    \rightarrowtail X.A \xrightarrow{\pi_A} X,\, X.Y \xrightarrow{\pi_Y} X\,) \\
    & \cong \mathop{\mathrm{colim}}\nolimits_{U \in \U} \cat{Set}/X.A(\,U
    \rightarrowtail X.A,\, \pi_A^\ast(X.Y) \xrightarrow{\smash{{\pi_A}^\ast\pi_Y}} X.A\,) \\
    & \cong \mathop{\mathrm{colim}}\nolimits_{U \in \U} \npr U
    Y(\pi_A(x)) =
    \upr {X.A} Y(\pi_A(x))\rlap{ ,}
  \end{align*}
  so that $y_{(A, \U)}$ corresponds to the composite~\eqref{eq:33} as desired.
\end{proof}

\subsection{Relation to model theory}
\label{sec:relat-model-theory}
Finally in this section, we relate Theorem~\ref{thm:6} to \emph{types}
in model theory. Recall that, if $\mathbb{T}$ is a (classical,
single-sorted) first-order theory, and we write
$\mathbb{T}[x_1, \dots, x_n]$ for the theory obtained by adjoining new
constants $x_1, \dots, x_n$ to $\mathbb{T}$, then a \emph{complete
  $n$-type} of $\mathbb{T}$ is a complete theory extending
$\mathbb{T}[x_1, \dots, x_n]$; equally, it is an ultrafilter on the
Lindenbaum--Tarski algebra $B_{\mathbb{T}[x_1, \dots, x_n]}$ of
sentences in $\mathbb{T}[x_1, \dots, x_n]$ identified up to provable
equivalence.

Each $n$-tuple of elements $\vec a = a_1, \dots, a_n$ in a
$\mathbb{T}$-model $A$ yields a complete $n$-type $\tau(\vec a)$,
namely the set of all sentences $\varphi(x_1, \dots, x_n)$ of
$\mathbb{T}[x_1, \dots, x_n]$ such that
$A \vDash \varphi(a_1, \dots, a_n)$; we say that the $n$-tuple
$\vec a$ \emph{realises} the type $\tau(\vec a)$. Many important
questions in model theory revolve around the existence, or otherwise,
of tuples of elements which realise a given type. As we now show, we
can obtain the realisation relation between $n$-types of $\mathbb{T}$
and $n$-tuples in a $\mathbb{T}$-model by applying Theorem~\ref{thm:6}
with $\C$ taken to be the \emph{classifying Boolean pretopos} of
$\mathbb{T}$.

The classifying Boolean pretopos of a first-order theory has a
characterising property analogous to that of the Lindenbaum--Tarski
algebra $B_\mathbb{P}$ of a (classical) propositional theory
$\mathbb{P}$. Indeed, there is a standard notion of \emph{model} of a
propositional theory $\mathbb{P}$ in a Boolean algebra $C$, and the
Lindenbaum--Tarski algebra $B_\mathbb{P}$ of all
$\mathbb{P}$-sentences modulo $\mathbb{P}$-provable equality, is the
\emph{universal} Boolean algebra which models $\mathbb{P}$. By this,
we mean that models of $\mathbb{P}$ in a Boolean algebra $C$ are in
bijection with Boolean homomorphisms $B_\mathbb{P} \rightarrow C$.

The classifying Boolean pretopos of a first-order theory $\mathbb{T}$
has a similar property: it is the universal Boolean pretopos which
models $\mathbb{T}$. We now make this precise. First, a
\emph{pretopos} is a category which is finitely complete, extensive
and also \emph{Barr-exact}~\cite{Barr1971Exact}, meaning that it has
well-behaved quotients of equivalence relations; while a pretopos is
\emph{Boolean} if it is so \emph{qua} extensive category. If $\C$ and
$\D$ are pretoposes, then a \emph{pretopos morphism}
$F \colon \C \rightarrow \D$ is a functor preserving finite limits,
finite coproducts and regular epimorphisms; we write
$\twocat{Pretop}(\C, \D)$ for the category of pretopos morphisms and
all natural transformations.

There is a standard notion of \emph{model} of a first-order theory
$\mathbb{T}$ in a Boolean pretopos $\C$---see, for example~\cite[\sec
D1.2]{Johnstone2002Sketches2}---and these comprise the objects of a
category $\mathbb{T}\text-\cat{Mod}_e(\C)$ whose maps are elementary
embeddings. If $\C$ and $\D$ are Boolean pretoposes, then any pretopos
morphism $F \colon \C \rightarrow \D$ preserves $\mathbb{T}$-models
and so induces a functor
$F_\ast \colon \mathbb{T}\text-\cat{Mod}_e(\C) \rightarrow
\mathbb{T}\text-\cat{Mod}_e(\D)$.

\begin{Defn}
  \label{def:21}
  A \emph{classifying Boolean pretopos} for a first-order theory $\mathbb{T}$
  is a Boolean pretopos $\cl$ endowed with a
  $\mathbb{T}$-model
  $\mathbf{G} \in \mathbb{T}\text-\cat{Mod}_e(\cl)$ such
  that, for any Boolean pretopos $\D$, the following functor
  is an equivalence:
  \begin{equation}\label{eq:43}
      \begin{aligned}
     \twocat{Pretop}(\cl, \D) & \rightarrow
     \mathbb{T}\text-\cat{Mod}_e(\D) \\
     F & \mapsto F_*(\mathbf G)\rlap{ .}
   \end{aligned}
 \end{equation}
\end{Defn}

To construct the classifying Boolean pretopos of a first-order theory
$\mathbb{T}$, we first form its \emph{first-order syntactic category}
$\C_\mathbb{T}^\mathrm{fo}$ whose objects are ``formal
$\mathbb{T}$-definable sets'' $\{ \vec x : \varphi(\vec x)\}$ (i.e.,
first-order formulae-in-context) and whose maps are
$\mathbb{T}$-provable equivalence classes of $\mathbb{T}$-provably
functional relations from $\{\vec x : \varphi(\vec x)\}$ to
$\{\vec y : \psi(\vec y)\}$; see, for example~\cite[\sec
D1.4]{Johnstone2002Sketches2}. The classifying Boolean pretopos $\cl$
is now obtained by freely adjoining finite coproducts and coequalisers
of equivalence relations to $\C_\mathbb{T}^\mathrm{fo}$ while
preserving its existing finite unions and image
factorisations---see~\cite[Proposition~A1.4.5~\&~Corollary~A3.3.10]{Johnstone2002Sketches}
for the necessary constructions. Alternatively, we can obtain $\cl$ as
the first-order syntactic category of $\mathbb{T}^\mathsf{eq}$, where
$(\thg)^\mathsf{eq}$ is Shelah's \emph{elimination of imaginaries};
this observation is due to Makkai and Reyes~\cite{Makkai1977First},
and is explained in detail in~\cite{Harnik2011Model}.

We will not describe the generic $\mathbb{T}$-model $\mathbf{G}$ in
$\cl$ explicitly; however, part of its genericity is the fact that
$\mathrm{Sum}_{\cl}(G)$ is the Lindenbaum--Tarski algebra
$B_{\mathbb{T}[x]}$. More generally $\mathrm{Sub}_{\cl}(G^n)$ is the
corresponding Lindenbaum--Tarski algebra
$B_{\mathbb{T}[x_1, \dots, x_n]}$. It follows that an ultrafilter on
$G^n \in \cl$ is exactly a complete $n$-type of $\mathbb{T}$. Now,
since $\D = \cat{Set}$ is a Boolean pretopos, we obtain
from~\eqref{eq:43} and Theorem~\ref{thm:6} a string of functors
\begin{equation*}
  \mathbb{T}\text-\cat{Mod}_e(\cat{Set}) \xrightarrow{\simeq}
  \twocat{Pretop}(\cl, \cat{Set}) \xrightarrow{\subseteq}
  \fc(\cl, \cat{Set}) \xrightarrow{\simeq}
  [\uf_{\cl}, \cat{Set}]
\end{equation*}
assigning to each (ordinary) $\mathbb{T}$-model $\mathbf M$ 
both a functor $M \colon \cl \rightarrow \cat{Set}$ and a
functor $\tilde M \colon \uf_{\cl} \rightarrow \cat{Set}$.
In this context, the passage from $M$ to $\tilde M$ was described
by Makkai in~\cite{Makkai1981The-topos}, who also observed its
model-theoretic import: it encodes the types realised by tuples
of elements of the model $\mathbf M$.
 
Indeed, the pretopos morphsim
$M \colon \cl \rightarrow \cat{Set}$
corresponding to the model $\mathbf M$ sends $G$ to the underlying set
$\abs {\mathbf M}$ of the model, sends $G^n$ to ${\abs {\mathbf M}}^n$
and sends $\varphi \in \mathrm{Sum}_{\cl}(G^n)$ to the set
$\{\,\vec m \in {\abs {\mathbf M}}^n : \mathbf M \vDash \varphi(\vec
m)\,\}$. Thus, by~\eqref{eq:17}, the value of the corresponding
$\tilde M \in [\uf_{\cl}, \cat{Set}]$ at a complete
$n$-type $\U$ is given by the set of $n$-tuples of elements of
$\mathbf M$ which realise the type $\U$:
\begin{equation*}
  \tilde M(X^n, \U) = \{ \vec m \in \abs{\mathbf M}^n : \varphi \in \U\,
  \Longleftrightarrow \,\mathbf M \vDash \varphi(\vec m)\}\rlap{ .}
\end{equation*}

\section{Second generalisation}
\label{sec:second-gener}
In this section, we give our second generalisation of
Theorem~\ref{thm:2}, which extends the first one to an equivalence
$\fc(\C, \E) \simeq [\uf_\C, \E]$, where $\C$ is extensive as before,
and now $\E$ is any locally connected Grothendieck topos. We then use
this result to reconstruct the indexed sum of ultrafilters, and in the
process of doing so construct interesting and natural bicategories of
\emph{ultramatrices} and \emph{ultraspans}. Finally, we describe how
this relates to the \emph{ultracategories} of~\cite{Makkai1987Stone}.

\subsection{Generalising the codomain category}
\label{sec:gener-codom-categ}
A locally small category $\E$ is a \emph{Grothendieck topos} if it is
equivalent to the category of sheaves on a small site; see, for
example,~\cite[Chapter~III]{Mac-Lane1994Sheaves}. Equivalently, by
Giraud's theorem (cf.~\cite[Appendix]{Mac-Lane1994Sheaves}), $\E$ is a Grothendieck topos just when it is
finitely complete, Barr-exact and infinitary extensive with a small
generating set.

In any Grothendieck topos $\E$, the functor
$\Gamma = \E(1, \thg) \colon \E \rightarrow \cat{Set}$ has a left
adjoint $\Delta \colon \cat{Set} \rightarrow \E$ which sends a set $X$
to the coproduct $\Sigma_{x \in X} 1$. We say that $\E$ is
\emph{locally connected} (or
\emph{molecular}~\cite{Barr1980Molecular}) if $\Delta$ has a further
left adjoint $\pi_0 \colon \E \rightarrow \cat{Set}$. For example,
by~\cite[p.414,~Ex.~7.6]{Deligne1972Theorie}, the topos of sheaves on
a space $X$ is locally connected just when $X$ is locally connected in
the usual sense; in this case,
$\pi_0 \colon \cat{Sh}(X) \rightarrow \cat{Set}$ sends a sheaf to the
set of connected components of the corresponding \'etale space over
$X$.

\begin{Thm}
  \label{thm:8}
  Let $\C$ be extensive and $\E$ a locally connected Grothendieck
  topos. The category $\fc(\C, \E)$ is equivalent to $[\uf_\C, \E]$
  via an equivalence whose two directions are given by the formulae~\eqref{eq:17}
  and~\eqref{eq:13}.\end{Thm}

In proving this, we require a straightforward generalisation of the
equivalence $[\el X, \cat{Set}] \simeq [\A, \cat{Set}] / X$
of~\eqref{eq:10} to an equivalence
\begin{equation}\label{eq:25}
  [\el X, \E] \simeq [\A, \E] / \Delta X
\end{equation}
for any Grothendieck topos $\E$. The generalisation makes use of the
fact that $\E$ is infinitary extensive; this means that it has all
small coproducts, and that for any set $X$, the coproduct functor
$\Sigma \colon \textstyle\prod_{x \in X} (\E / A_x) \rightarrow \E /
(\Sigma_{i \in X} A_x)$ is an equivalence of categories. Taking each
$A_x$ to be terminal and using $\E / 1 \cong \E$, we deduce that
\begin{equation*}
   \Sigma \colon \E^X \rightarrow \E / \Delta X
\end{equation*}
is an equivalence for any set $X$, with pseudoinverse given
(necessarily) by pullback along the coproduct coprojections. By using
this equivalence in place of~\eqref{eq:26}, we may
generalise~\eqref{eq:10} to the desired equivalence~\eqref{eq:25}.
Much as before, $Y \in [\el X, \E]$ is sent to
$\pi \colon \int\!Y \rightarrow \Delta X$ whose component
$\pi_A \colon (\int\!Y)A \rightarrow \Delta XA$ is the coproduct of
the family of maps $(Y(A,x) \rightarrow 1)_{x \in XA}$; while
conversely, $p \colon E \rightarrow X$ in $[\A, \E]/\Delta X$ is sent
to $\tilde E \in [\el X, \E]$ wherein $\tilde E(A,x)$ is the pullback
of $p_A \colon EA \rightarrow \Delta XA$ along
$\Delta x \colon \Delta 1 \rightarrow \Delta XA$.

\begin{proof}[Proof of Theorem~\ref{thm:8}]
  Since
  $\pi_0 \dashv \Delta \dashv \Gamma \colon \E \rightarrow\cat{Set}$,
  both $\pi_0$ and $\Delta$ preserve finite coproducts, so inducing an
  adjunction
  $\pi_0 \circ (\thg) \dashv \Delta \circ (\thg) \colon \fc(\C,
  \cat{Set}) \rightarrow \fc(\C, \E)$, whose right adjoint must send
  the terminal object $\beta \in [\C, \cat{Set}]$ to a terminal object
  $\Delta \beta \in \fc(\C, \E)$. As $\C$ and $\E$ are extensive, it
  follows from this and Lemma~\ref{lem:2}~that
  $\fc(\C, \E) \cong [\C, \E] \,/_{\mathrm{sc}}\, \Delta \beta$ where
  to the right we have the full subcategory of the slice category on
  the sum-cartesian transformations; recall that \emph{sum-cartesian}
  means that the naturality squares at coproduct coprojections are
  pullbacks.

  Using~\eqref{eq:25} we have, like before, an equivalence
  $[\C, \E] / \Delta \beta \simeq [\el \beta, \E] \cong [\ue_\C, \E]$;
  and, like before, an object $p \colon E \rightarrow \Delta \beta$ to
  the left is sum-cartesian just when the corresponding
  $\tilde E \in [\ue_\C, \E]$ inverts the class $\M_\Sigma$ of
  continuous coproduct coprojections. Thus
  $\fc(\C, \E) \cong [\C, \E] \,/_{\mathrm{sc}}\, \Delta \beta \simeq
  [\uf_\C, \E]$ as desired.

  It remains to show that the two directions of the equivalence are
  given as in~\eqref{eq:17} and~\eqref{eq:13}. In the latter case this
  is clear from the construction using~\eqref{eq:25}. In the other
  direction, if $A \in \fc(\C, \E)$, then the corresponding
  $\smash{\tilde A} \in [\ue_\C, \E]$ has its value at $(X, \U)$ given
  by the pullback to the left in:
  \begin{equation*}
    \cd{
      {\tilde A(X, \U)} \ar[r]^-{} \ar[d]_{} &
      {AX} \ar[d]^{\tau_X} \\
      {\Delta 1} \ar[r]^-{\Delta \U} &
      {\Delta \beta X}
    } \qquad \qquad 
    \cd{
      {AU} \ar@{ >->}[r]^-{Am} \ar[d]_{\tau_U} &
      {AX} \ar[d]^{\tau_X} \\
      {\Delta \beta U} \ar@{ >->}[r]^-{\Delta \beta m} &
      {\Delta \beta X}\rlap{ ,}
    }
  \end{equation*}
  where $\tau \colon A \rightarrow \Delta \beta$ is induced by
  terminality of $\Delta \beta$ in $\fc(\C, \E)$. Note, however, that
  $\U \colon 1 \rightarrow \beta X$ in $\cat{Set}$ is the meet of the
  subobjects $(\beta U \rightarrowtail \beta X)_{U \in \U}$. Since
  $\Delta$ is a right adjoint, it preserves meets, as does pullback
  along $\tau_X$; consequently, $\tilde A(X, \U)$ is the meet of the
  subobjects
  $(\tau_X^{-1}(\Delta \beta U) \rightarrowtail AX)_{U \in \U}$. But
  since $\tau$ is sum-cartesian by Lemma~\ref{lem:2}, the square right
  above is a pullback for any $m \colon U \rightarrowtail X$ in $\U$,
  and so we conclude that $\tilde A$ is given as in~\eqref{eq:17} by
  \begin{equation*}
    \tilde A(X, \U) \cong \bigcap_{U \in \U} AU \subseteq AX\rlap{ .}\qedhere
  \end{equation*}
\end{proof}

\subsection{Ultramatrices, ultraspans and the relation to indexed
  sums}
\label{sec:ultr-ultr-relat}

We now wish to describe how this result recaptures the indexed sum of
ultrafilters. In fact, we will do something slightly more general to
draw as perfect an analogy as possible with Proposition~\ref{prop:3}.
The first step there was to transport the strict monoidal structure on
the category $\fc(\cat{Set}, \cat{Set})$ to obtain a monoidal
structure on the equivalent $[\uf, \cat{Set}]$. The analogue here is
to transport the compositional structure of a $2$-category of
finite-coproduct-preserving functors along equivalences of each of its
hom-categories to obtain an equivalent
\emph{bicategory}~\cite{Benabou1967Introduction}.

\begin{Defn}
  \label{def:13}
  \begin{enumerate}[(i),itemsep=0.25\baselineskip]
  \item 
  The $2$-category $\twocat{FC}_{\cat{Set}}$ has sets $X,Y,Z,\dots$
  as objects; hom-categories given by
  $\twocat{FC}_\cat{Set}(X,Y) = \fc(\cat{Set}^X, \cat{Set}^Y)$; and
  composition given by the usual composition of functors and natural
  transformations.

\item   The bicategory $\twocat{UEsp}$ of \emph{ultrafilter species} has
  sets as objects; hom-categories
  $\twocat{UEsp}(X,Y) = [\uf_{\cat{Set}^X}, \cat{Set}^Y]$; and
  composition obtained from that of $\twocat{FC}_{\cat{Set}}$ by
  transporting across the equivalences
  $\fc(\smash{\cat{Set}^X}, \smash{\cat{Set}^Y}) \simeq
  [\uf_{\cat{Set}^X}, \smash{\cat{Set}^Y}]$.
\end{enumerate}
\end{Defn}

The nomenclature ``ultrafilter species'' echoes Joyal's notion of a
species of structures (\emph{esp\'eces de
  structures}~\cite{Joyal1986Foncteurs}), and its generalisation
in~\cite{Fiore2008The-cartesian} to a bicategory $\twocat{Esp}$ of
\emph{generalised species of structures}. We will not labour the
comparison, but suffice it to say that in both bicategories,
composition is given by a substitution formula, which in the case of
$\twocat{Esp}$ is given by equation~(9)
of~\cite{Fiore2008The-cartesian}, and for $\twocat{UEsp}$ is given
by a suitable generalisation of~\eqref{eq:31}.

In Proposition~\ref{prop:3}, we reconstructed the tensor product of
ultrafilters by showing the representables in $[\uf, \cat{Set}]$ to be
closed under the tensor product. To reconstruct the indexed sum of
ultrafilters, we will similarly show that \emph{pointwise representable}
$1$-cells in $\twocat{UEsp}$ are closed under composition. Here,
$F \in \twocat{UEsp}(X,Y) = [\uf_{\cat{Set}^X}, \cat{Set}^Y]$ is
pointwise representable if each functor
$F(\thg)(y) \colon \uf_{\cat{Set}^X} \rightarrow \cat{Set}$ is
representable. The subcategory of pointwise representable functors is
equivalent (via pointwise Yoneda) to the category
$(\uf_{\cat{Set}^X})^Y$, and so a typical pointwise representable
$1$-cell is presented by a $Y$-indexed family of pairs
$(M_y \in \cat{Set}^X, \U_{y} \in \beta(X.M_{y}))$.
In fact, we prefer to think of these data in either one of the
following two alternative ways.
\begin{Defn}
  \label{def:18}
  Let $X$ and $Y$ be sets.
  \begin{enumerate}[(i),itemsep=0.25\baselineskip]
  \item An \emph{ultramatrix} from $X$ to $Y$ is a pair $(M, \U)$
    composed of a matrix of sets $M \in \cat{Set}^{X \times Y}$ together
    with a $Y$-indexed family of ultrafilters $\U_{y}$ on each column
    sum $M_y \defeq \nsig X M(x,y)$.
  \item An \emph{ultrafamily}
  $(g, \U) \colon M \rightsquigarrow Y$ is a function
  $g \colon M \rightarrow Y$ together with an ultrafilter $\U_{y}$ on
  each fibre $g^{-1}(y)$. An \emph{ultraspan} from $X$ to $Y$ is a
  span with left leg a function and right leg an ultrafamily:
  \begin{equation}\label{eq:30}
    \cd[@-0.5em]{ & M \ar[dl]_-{f} \ar@{~>}[dr]^-{(g, \U)} \\
      X & & Y\rlap{ .}
    }
  \end{equation}
  \end{enumerate}
\end{Defn}

It is easy to see using~\eqref{eq:26} that both ultramatrices and
ultraspans from $X$ to $Y$ correspond to pointwise representables in
$\twocat{UEsp}(X,Y)$, and so to certain finite-coproduct-preserving
functors $\cat{Set}^X \rightarrow \cat{Set}^Y$. As in the
introduction, we may call these \emph{generalised ultraproduct
  functors}. Using Proposition~\ref{prop:6}, we see that, one the one
hand, the generalised ultraproduct functor
$\cat{Set}^X \rightarrow \cat{Set}^Y$ encoded by the ultramatrix
$(M, \U)$ is given by:
\begin{equation}\label{eq:18}
  (\,H(x) \mid x \in X\,) \qquad \mapsto \qquad 
  \big(\,(\Pi_{\U_{y}}
    (x,m)\,\mathord\in\,M_{y})\,\!H(x) \mid y \in Y\,\big)\rlap{ .}
\end{equation}
On the other hand, the ultraspan $(f, (g, \U)) \colon X \rightarrow Y$
encodes the functor
\begin{equation*}
  \cat{Set}^X \xrightarrow{\cat{Set}^f} \cat{Set}^M
  \xrightarrow{\Pi_{(g,\U)}} \cat{Set}^Y
\end{equation*}
where $\Pi_{(g,\U)}$ is given by ``ultraproduct on each fibre'';
i.e., its $y$-component is given by restriction
$\cat{Set}^M \rightarrow \cat{Set}^{g^{-1}y}$ followed by ultraproduct
$\Pi_{\U_{y}} \colon \cat{Set}^{g^{-1}y} \rightarrow \cat{Set}$.

The next two definitions are intended to describe how pointwise
representable $1$-cells in $\twocat{UEsp}$ compose in terms of the
representing ultramatrices or ultraspans.

\begin{Defn}
  \label{def:17}
  If $(M, \U)$ and $(N, \V)$ are ultramatrices from $X$ to $Y$ and
  from $Y$ to $Z$, then their \emph{composition} is the
  ultramatrix $(N \cdot M, \V \cdot \U)$ from $X$ to $Z$ whose first
  component is given by the usual matrix multiplication:
  \begin{equation*}
    (N \cdot M)(x,z) = \nsig[y] Y (N(y,z) \times M(x,y))\rlap{ .}
  \end{equation*}
  As for the second component, note that for each $z \in Z$ we have
  an isomorphism
  \begin{equation}\label{eq:34}
    \nsig[(y,n)]{N_z} M_{y} \cong (N \cdot M)_z
  \end{equation}
  sending $(y,n,x,m)$ to $(x,y,n,m)$. We can therefore define the
  ultrafilter $(\V \cdot \U)_z$ on $(N \cdot M)_z$ to be the transport
  across~\eqref{eq:34} of the ultrafilter on $\nsig[(y,n)]{N_z} M_{y}$
  given by the indexed sum
  $(\Sigma_{\V_z} (y,n)\! \in\! N_z)\,\!\U_{y}$. 
\end{Defn}
\begin{Defn}
  \label{def:19}
  Given ultraspans $(f, (g, \U)) \colon X \rightarrow Y$ and
  $(h,(k, \V)) \colon Y \rightarrow Z$, their \emph{composition}
  is the ultraspan whose legs are given by the outer composites in:
  \begin{equation*}
    \cd[@-3.5em@!]{
      & & M \times_Y N
      \ar[dl]_-{p}
      \ar@{~>}[dr]^-{(q, \W)} \pullbackcorner[d] \\
      & M \ar[dl]_-f \ar@{~>}[dr]^-{(g, \U)} & & N \ar[dl]^-{h}
      \ar@{~>}[dr]^-{\scriptstyle (k, \V)} \\
      X && Y && Z\rlap{ .}
    }
  \end{equation*}
  Here, $p$ and $q$ constitute a pullback of $g$ and $h$ in
  $\cat{Set}$. To the top right, the pullback ultrafamily
  $(q, \W) \colon M \times_Y N \rightsquigarrow N$ has $\W_n$ given
  by the transport of $\U_{hn}$ across the isomorphism
  $g^{-1}(hn) \cong q^{-1}(n)$. Finally, the composite
  $(kq, \V \W)$ of the ultrafamilies $(q, \W)$ and $(k, \V)$ has
  $(\V\W)_z$ given by the transport of
  $(\Sigma_{\V_z} n\! \in\! k^{-1}z)\,\!\W_{n}$ across the
  isomorphism
  $(\Sigma n\! \in\! k^{-1}z)\,\!q^{-1}n \cong (kq)^{-1}z$.
\end{Defn}

The validity of these descriptions is confirmed by:
\begin{Prop}\label{prop:15}
  The pointwise representable $1$-cells in $\twocat{UEsp}$ are
  closed under composition, with the induced composition on
  ultramatrices and ultraspans given as in Definition~\ref{def:17} and
  Definition~\ref{def:19} respectively.
\end{Prop}
\begin{proof}
  The identity $1$-cells in $\twocat{UEsp}$ are easily seen to be
  pointwise representable. As for binary composition, the composition
  laws in Definitions~\ref{def:17} and~\ref{def:19} correspond
  under~\eqref{eq:26}, so that it suffices to check the claim on
  ultramatrices. So let $F \in \twocat{UEsp}(X,Y)$ and
  $G \in \twocat{UEsp}(Y,Z)$ be represented by the respective
  ultramatrices $(M, \U)$ and $(N,\V)$. By~\eqref{eq:18}, the
  corresponding generalised ultraproduct functors
  $\int\!F \in \fc(\cat{Set}^X, \cat{Set}^Y)$ and
  $\int\!G \in \fc(\cat{Set}^Y, \cat{Set}^Z)$ have respective actions
  on objects
  \begin{equation*}
    \textstyle(\int\!F)(H)(y) = (\Pi_{\U_{y}}
    (x,m)\,\mathord\in\,M_{y})\,\!H(x) \ \text{ and } \ 
    \textstyle(\int\!G)(K)(z) = (\Pi_{\V_{z}}
    (y,n)\,\mathord\in\,N_{z})\,\!K(y)\rlap{ .}
  \end{equation*}
  Therefore 
  $\int GF \cong \int\!G \circ \int\!F \colon \cat{Set}^X \rightarrow \cat{Set}^Z$
  satisfies
  \begin{align*}
    \textstyle(\int\!GF) (H)(z) & \cong (\Pi_{\V_{z}}
    (y,n)\,\mathord\in\,N_{z})(\Pi_{\U_{y}}
    (x,m)\,\mathord\in\,M_{y})\,\!H(x)\\
    &\cong (\Pi_{(\Sigma_{\V_z} (y,n)\! \in\! N_z)\,\!\U_{y}}
    (y,n,x,m)\,\mathord\in\,\nsig[(y,n)]{N_z} M_{y})\,\!H(x)\\
    &\cong (\Pi_{(\V \cdot \U)_z}
    (x,y,n,m) \,\mathord\in\,(N \cdot M)_z)\,\!H(x)\rlap{ ,}
  \end{align*}
  using Proposition~\ref{prop:12} and the definition of
  $(\V \cdot \U)_z$. Thus, by~\eqref{eq:18} again, the pointwise
  representability of $GF$ is witnessed by the ultramatrix
  $(N \cdot M, \V \cdot \U)$.
\end{proof}

It follows from this result that there are bicategories
$\twocat{UMtx}$ (resp., $\twocat{US}$) in which objects are sets;
$1$-cells are ultramatrices (resp., ultraspans) composing as in
Definition~\ref{def:17} (resp., Definition~\ref{def:19}); and
$2$-cells are determined by the requirement that each bicategory be
biequivalent to the locally full sub-bicategory of $\twocat{UEsp}$
on the pointwise representable $1$-cells.

It remains to show that the composition laws in $\twocat{UMtx}$ and
$\twocat{US}$ allow us to reconstruct the indexed sum of
ultrafilters, so fulfilling the objective of this section. This is
easiest to see in the case of $\twocat{US}$. Suppose that we are
given a set $X$ equipped with an ultrafilter $\U$ and an $X$-indexed
family of sets $Y(x)$ each equipped with an ultrafilter $\V(x)$. We
can represent these data as a pair of composable ultraspans as to the
left in:
\begin{equation*}
  \cd[@-1.5em@!]{
    & X.Y \ar[dl]_-1 \ar@{~>}[dr]^-{(\pi_Y, \V)} & & X \ar[dl]^-{1}
    \ar@{~>}[dr]^-{(!, \U)} \\
    X.Y && X && 1
  } \qquad \qquad 
  \cd[@-1.5em@!]{
    & X.Y \ar[dl]_-1 \ar@{~>}[dr]^-{(!, \U.\V)} \\
    X.Y && 1\rlap{ ,}
  }
\end{equation*}
whose composite encodes the indexed sum $\U. \V$
as right above.

\begin{Rk}
  \label{rk:9}
  In Remark~\ref{rk:4} above, we explained how categories
  \emph{enriched} over the monoidal category $[\uf, \cat{Set}]$ can be
  thought of as ordinary categories endowed with abstract ultrapower
  functors. It is possible to extend this so as to capture
  categories endowed with abstract \emph{ultraproduct} functors by
  using the theory of \emph{categories enriched in a bicategory} as
  in~\cite{Walters1981Sheaves}.

  Rather than $[\uf, \cat{Set}]$-enriched categories with copowers by
  representables, we consider $\twocat{UEsp}$-enriched categories with
  copowers by pointwise representable $1$-cells. We might guess that
  such enriched categories correspond to ordinary categories $\C$
  admitting abstract ultraproduct functors
  $\Pi_\U \colon \C^X \rightarrow \C$. The reality is slightly more
  subtle; while some details still require sorting out, it appears
  that the $\twocat{UEsp}$-enriched categories with copowers as above
  correspond to \emph{$\cat{Set}$-indexed prestacks}---i.e.,
  pseudo-functors
  $\mathbb C \colon \cat{Set}^\mathrm{op} \rightarrow \twocat{CAT}$
  satisfying a descent condition---equipped with suitably coherent
  abstract ultrapower functors
  $\Pi_{(f,\U)} \colon \mathbb{C}_X \rightarrow \mathbb{C}_Y$ for each
  ultrafamily $(f, \U) \colon X \rightsquigarrow
  Y$.

  While the details must await a further paper, these observations
  draw an interesting link to Makkai's
  \emph{ultracategories}~\cite{Makkai1987Stone, Makkai1993Duality}. As
  in the introduction, an ultracategory is a category endowed with
  abstract ultraproduct structure, as well as interpretations for any
  \emph{ultramorphism}, i.e., ``definable map between ultraproducts''.
  Makkai's main result is that the ultracategory structure on the
  category of models of a coherent theory $\mathbb{T}$ in either
  intuitionistic or classical first-order logic is sufficient to
  reconstruct $\mathbb{T}$ to within Morita equivalence; more
  precisely, it suffices to reconstruct the classifying pretopos of
  $\mathbb{T}$.
  
  These results of Makkai were given new proofs by Lurie
  in~\cite{Lurie2018Ultracategories}, with a significantly simplified
  definition of what it means for an ultracategory to admit
  interpretations of any ultramorphism; and although we have not yet
  completed the analysis, it seems that this additional structure is
  exactly what $\twocat{UEsp}$-enrichment provides besides the
  existence of abstract ultraproduct functors. In future work we hope
  to investigate this further with a view to giving a purely
  enriched-categorical proof of Makkai's reconstruction result.
\end{Rk}

\section{Locally connected classifying toposes}

In Section~\ref{sec:relat-model-theory}, we discussed the algebraic
semantics for \emph{classical} propositional theories in Boolean
algebras, and for first-order theories in Boolean pretoposes. There
are similar semantics for certain fragments of intuitionistic logic:
in particular, the \emph{coherent} fragment, which allows only the
logical connectives $\exists, \vee, \wedge, \top, \bot$, has an
algebraic semantics in distributive lattices (for the propositional
case) and in pretoposes (for the first-order case); see, for
example~\cite{Makkai1977First}.

Now for classical propositional theories, there is a restricted
semantics valued not in Boolean algebras but \emph{complete atomic}
Boolean algebras. Rather than the Lindenbaum--Tarski algebra
$B_\mathbb{P}$, the appropriate universal model in this context is the
power-set algebra $\P(T_\mathbb{P})$ on the set $T_\mathbb{P}$ of
complete theories extending $\mathbb{P}$, with the universal valuation
of $\mathbb{P}$ in $\P(T_\mathbb{P})$ sending each primitive
proposition $\varphi$ to the set of complete extensions which validate
it. Since $T_\mathbb{P}$ is equally the set $\beta(B_\mathbb{P})$ of
ultrafilters on $B_\mathbb{P}$, the passage from $B_\mathbb{P}$ to
$\P(T_\mathbb{P})$ can be understood in terms of the lattice-theoretic
construction of \emph{canonical extension}~\cite{Jonsson1951Boolean}:
in general, for a Boolean algebra $B$, its canonical extension is the
power-set algebra $B^\delta \defeq \P(\beta B)$.

Similarly, for coherent intuitionistic propositional theories, the
semantics in distributive lattices restricts to one valued in
completely distributive algebraic lattices; these are equally well the
down-set lattices of posets, and posets are the same thing as
symmetric, transitive Kripke frames, so this is really another view on
the \emph{Kripke semantics} of intuitionistic
logic~\cite[\sec2.5]{Troelstra1988Constructivism}. 
In this context, the classifying distributive lattice $D_\mathbb{P}$
of a theory $\mathbb{P}$ is replaced by its canonical extension
$D_\mathbb{P}^\delta$, constructible as in~\cite{Gehrke1994Bounded} as
the downset-lattice of the poset of prime filters in $D_\mathbb{P}$
(ordered by reverse inclusion);
this is an algebraic reformulation of the construction of the
canonical Kripke
model~\cite[Definition~2.6.4]{Troelstra1988Constructivism} of an
intuitionistic propositional theory.

An obvious question is whether the alternate semantics detailed above
lift from propositional to first-order logic. A variety of positive
answers have been given to this question. One positive answer is given
in~\cite{Makkai1982Full}; there, Makkai introduces the notion of a
\emph{prime-generated topos}, as a Grothendieck topos whose subobject
lattices are completely distributive algebraic\footnote{This is not
  the original definition of prime-generation
  from~\cite{Makkai1981The-topos}, but an equivalent one
  from~\cite[\sec 3]{Barr1987On-representations}.}. He then considers
semantics for a coherent intuitionistic first-order theory
$\mathbb{T}$ valued in prime-generated toposes, and constructs the
universal such model in what he calls the \emph{topos of types} of
$\mathbb{T}$. This is obtained from the classifying pretopos
$\C = \cl$ of $\mathbb{T}$ by a ``categorified canonical extension'',
obtained by first forming a category $\P\F_\C$ of \emph{prime filters}
in $\C$, defined similarly to our $\uf_\C$, and then taking a suitable
sheaf subcategory $\tau(\cl)$ of $[\P\F_\C^\mathrm{op}, \cat{Set}]$.
The name ``topos of types'' for $\tau(\cl)$ derives from the fact
that, as Makkai puts it, it gives `a reasonable codification of the
`discrete' (non topological) syntactical structure of types of the
theory''~\cite[p.196]{Makkai1981The-topos}. The idea that the passage
$\cl \mapsto \tau(\cl)$ should be seen as a kind of canonical
extension was made precise by Coumans
in~\cite{Coumans2012Generalising}.

In this section, we take a slightly different view on lifting
canonical extension to the first-order coherent context. In Makkai's
approach, the distributive subobject lattices of $\cl$ are completed
via canonical extension to the completely distributive algebraic
subobject lattices of $\tau(\cl)$. For us, the focus will instead be
on the Boolean algebras of coproduct summands in $\cl$---corresponding
logically to decidable predicates in $\mathbb{T}$---which will be
completed via canonical extension to complete atomic Boolean algebras
of coproduct-summands. The property of having complete atomic Boolean
algebras of coproduct-summands is,
by~\cite[Theorem~15]{Barr1980Molecular}, characteristic of locally
connected Grothendieck toposes, and so what we aim to describe in the
\emph{locally connected classifying topos} of a given pretopos $\C$.
The existence of this follows from results
of~\cite{Funk1999The-locally}, but the description there is rather
inexplicit; we aim to give a concrete construction in terms of sheaves
on the category of ultrafilters, and to compare this to the toposes of
types of Makkai and others.

\subsection{The lextensive case}
\label{sec:lextensive-case}
In this section, as a warm-up to our main result, we construct the
locally connected classifying topos of a small \emph{lextensive}
category---that is, a category which is both finitely complete and
extensive.

We first make precise what we mean by this.
Recall that a \emph{geometric morphism} $f \colon \E \rightarrow \F$
between toposes is an adjoint pair of functors
$f^\ast \dashv f_\ast \colon \E \rightarrow \F$ such that $f^\ast$
(the \emph{inverse image functor}) preserves finite limits. We write
$\twocat{LCGTop}$ for the $2$-category of locally connected
Grothendieck toposes, geometric morphisms and natural transformations
$f^\ast \Rightarrow g^\ast$, and write $\twocat{Lext}$ for the
$2$-category of lextensive categories, lextensive
functors (i.e., ones preserving finite limits and finite coproducts)
and arbitrary natural transformations. As every locally connected
Grothendieck topos and every inverse image functor between such is
lextensive, we have a forgetful $2$-functor
$\twocat{LCGTop}^\mathrm{op} \rightarrow \twocat{Lext}$.

\begin{Defn}
  \label{def:20}
  A \emph{locally connected classifying topos} for an extensive
  category $\C$ is a left biadjoint at $\C$ for the forgetful
  $2$-functor
  $\twocat{LCGTop}^\mathrm{op} \rightarrow \twocat{Lext}$.
\end{Defn}

Here, and in what follows, when we speak of a \emph{left biadjoint at
  $X$} for a $2$-functor
$U \colon \twocat{A} \rightarrow \twocat{B}$, we mean a
birepresentation (in the sense of~\cite{Street1980Fibrations}) for the
$2$-functor
$\twocat B(X, U\thg) \colon \twocat A \rightarrow \twocat{CAT}$.
More concretely, then, a locally connected classifying topos for the lextensive $\C$ comprises
a locally connected Grothendieck topos $\cat{Lc}(\C)$ and a lextensive functor
$\eta \colon \C \rightarrow \cat{Lc}(\C)$ which is \emph{universal}
in the sense that, for each locally connected Grothendieck topos $\E$,
we have an
equivalence:
\begin{equation}
  \label{eq:46}
    \twocat{LCGTop}(\E, \cat{Lc}(\C)) \simeq \twocat{Lext}(\C, \E) \\
\end{equation}
induced by the assignment $f \mapsto f^\ast \circ \eta$.

Our goal is to give an explicit construction of a locally connected
classifying topos for any small lextensive category $\C$. For this, we
require the result sometimes known as \emph{Diaconescu's theorem}; it
can be found proved in, for
example,~\cite[Theorem~VII.7.2]{Mac-Lane1994Sheaves}.

\begin{Prop}
  \label{prop:11}
  If $\A$ is a small category, then the presheaf topos
  $[\A^\mathrm{op}, \cat{Set}]$ classifies flat functors out of $\A$.
  More precisely, for each Grothendieck topos $\E$, the assignment
  $f \mapsto f^\ast \circ y$ induces an equivalence of categories
\begin{equation}
  \label{eq:53}
    \twocat{GTop}(\E, [\A^\mathrm{op}, \cat{Set}]) \simeq \twocat{Flat}(\A, \E) \rlap{ .}
\end{equation}
\end{Prop}
Here, we define $\twocat{Flat}(\A, \E)$ as the full subcategory of
$[\A, \E]$ on the flat functors, but for this we should clarify what
``flat'' means. One definition is that $F \colon \A \rightarrow \E$ is
flat just when
$\mathrm{Lan}_{y} F \colon [\A^\mathrm{op}, \cat{Set}] \rightarrow
\E$, its left Kan extension along the Yoneda embedding of
$[\A^\mathrm{op}, \cat{Set}]$, preserves finite limits; this is a
general categorical definition which makes sense for any small $\A$
and cocomplete $\E$. On the other hand, when $\E$ is a Grothendieck
topos as above, a more explicit characterisation is possible which
generalises a well-known characterisation when $\E = \cat{Set}$.

Given $F \in [\A, \E]$, we write $\el F$ for the
\emph{category of elements} of $F$: the internal category in $\E$ with
underlying graph
\begin{equation*}
  \cd{
   \textstyle\sum_{a,b\in \A}\sum_{f \in \A(a,b)}
   Da \ar@<3pt>[r]^-{s}
   \ar@<-3pt>[r]_-{t} &
  \textstyle\sum_{a \in \A} Da 
  }
\end{equation*}
where $s$ maps the $(a,b,f)$-summand to the $a$-summand via $1_{Da}$,
and where $t$ maps the $(a,b,f)$-summand to the $b$-summand via $Df$.
There is a standard notion---see, for
example~\cite[Definition~B2.6.2]{Johnstone2002Sketches}---of what it
means for an internal category $\mathbb{C}$ in a topos to
be \emph{cofiltered}; in the internal language of the topos, it says
that ``every finite diagram in $\mathbb{C}$ has a cocone under it''.
The key result we will need is the following; for a proof,
see~\cite[Theorem~B3.2.7]{Johnstone2002Sketches}.
\begin{Prop}
  \label{prop:20}
  If $\A$ is a small category and $\E$ is a Grothendieck topos, then
  $F \colon \A \rightarrow \E$ is flat if and only if the internal
  category $\el F$ in $\E$ is cofiltered.
\end{Prop}

With these preliminaries in place, we can now give:

\begin{Prop}
  \label{prop:19}
  If $\C$ is small and lextensive, then 
  $[{\uf_\C}^\mathrm{op}, \cat{Set}]$ is a locally connected
  classifying topos for $\C$.
\end{Prop}

\begin{proof}
  Like any presheaf topos, $[{\uf_\C}^\mathrm{op}, \cat{Set}]$ is
  locally connected. For the classifying property, we must exhibit
  equivalences
  $\twocat{LCGTop}(\E, [{\uf_\C}^\mathrm{op}, \cat{Set}]) \simeq
  \twocat{Lext}(\C, \E)$, pseudonaturally in $\E$, which we will do
  by composing pseudonatural equivalences:
  \begin{equation}\label{eq:49}
    \twocat{LCGTop}(\E, [{\uf_\C}^\mathrm{op}, \cat{Set}]) \xrightarrow{\simeq}
    \twocat{Flat}(\uf_\C, \E) \xrightarrow{\simeq} \twocat{Lext}(\C, \E)\rlap{ .}
  \end{equation}
  The first of these is~\eqref{eq:53}. As for the second, we have by
  Theorem~\ref{thm:8} that
  \begin{equation}\label{eq:20}
    [\uf_\C, \E] \simeq \fc(\C, \E)
  \end{equation}
  for any locally connected Grothendieck topos $\E$, and by
  considering the explicit formula~\eqref{eq:13} for the rightward
  direction, we see that these equivalences are pseudonatural in
  inverse image functors. We will thus have 
  the desired pseudonatural equivalence if we can show that,
  in~\eqref{eq:20}, the flat functors on the left-hand side correspond
  to the finite-limit-preserving ones on the right.
  
  Towards this goal, we recall from Definition~\ref{def:7} the
  category $\ue_\C$ of which $\uf_\C$ is a localisation, and consider
  the span $\pi \colon \C \leftarrow \ue_\C \rightarrow \uf_\C \colon \iota$
  whose two legs are the forgetful functor and the localisation
  functor respectively. It is easy to see from the
  formula~\eqref{eq:13} that the left-to-right direction
  of~\eqref{eq:20} sends $B \colon \uf_\C \rightarrow \E$ to its image
  under the composite functor
  \begin{equation*}
    [\uf_\C, \E] \xrightarrow{\iota^\ast} [\ue_\C, \E]
    \xrightarrow{\mathrm{Lan}_\pi} [\C, \E]\rlap{ .}
  \end{equation*}
  It therefore suffices to prove that:
  
  (i) \textbf{$B \colon \uf_\C \rightarrow \E$ is flat if and only if
    $B \iota \colon \ue_\C \rightarrow \E$ is flat}. We saw above that
  $\uf_\C \cong \ue_\C[\M_\Sigma^{-1}]$, the localisation at the class
  of continuous coproduct coprojections. Since $\M_\Sigma$ is a
  pullback-stable, composition-closed class of arrows, there is (for
  example by~\cite[Proposition~C2.1.9]{Johnstone2002Sketches2}) a
  Grothendieck topology $J$ on $\ue_\C$ whose covering sieves are
  those which contain any map in $\M_\Sigma$. The condition for a
  functor $F \colon \ue_\C^\mathrm{op} \rightarrow \cat{Set}$ to be a
  $J$-sheaf is now precisely the condition that it inverts each
  $m \in \M_\Sigma$, and so we may identify
  $[{\uf_\C}^\mathrm{op}, \cat{Set}]$ with $\cat{Sh}(\ue_\C)$, and the
  sheafification adjunction with the adjunction
  $\mathrm{Lan}_\iota \dashv \iota^\ast \colon [{\uf_\C}^\mathrm{op},
  \cat{Set}] \rightarrow [{\ue_\C}^\mathrm{op}, \cat{Set}]$;
    so in particular, $\mathrm{Lan}_\iota$ preserves finite limits. We
    now use this to prove the claim. Note that
    $\mathrm{Lan}_y (B\iota) \cong (\mathrm{Lan}_y B) \circ
    \mathrm{Lan}_\iota$, so that if $B$ is flat, then so too is
    $B\iota$. On the other hand, since
    $\mathrm{Lan}_\iota \circ \iota^\ast \cong 1$, we have that
    $\mathrm{Lan}_y B \cong \mathrm{Lan}_y (B \iota) \circ \iota^\ast$
    so that if $B \iota$ is flat then so too is $B$.

    (ii) \textbf{$F \colon \ue_\C \rightarrow \E$ is flat if and only
      if $\mathrm{Lan}_\pi F \colon \C \rightarrow \E$ preserves
      finite limits}. Since the value of
    $\mathrm{Lan}_\pi F \colon \C \rightarrow \E$ at $X$ is
    $\sum_{\U \in \beta X} F(X, \U)$, it is an easy 
    calculation to see that the internal categories $\el F$ and
    $\el (\mathrm{Lan}_\pi F)$ are isomorphic. So $F$ is flat if and
    only if $\mathrm{Lan}_\pi F$ is flat. But since $\C$ admits all
    finite limits, $\mathrm{Lan}_\pi F$ is flat if and only if it is
    finite-limit-preserving; see, for
    example~\cite[Lemma~B3.2.5]{Johnstone2002Sketches}.
\end{proof}
By tracing the identity geometric morphism on
$[{\uf_\C}^\mathrm{op}, \cat{Set}]$ through this proof, we see
that the universal lextensive functor
$\eta \colon \C \rightarrow [{\uf_\C}^\mathrm{op}, \cat{Set}]$ is the
image under~\eqref{eq:13} of the Yoneda embedding
$\uf_\C \rightarrow [{\uf_\C}^\mathrm{op}, \cat{Set}]$, and so given
by:
\begin{equation*}
  \eta(X) = \textstyle\sum_{\U \in \beta X} y_{(X, \U)}\rlap{ .}
\end{equation*}

\subsection{The pretopos case}
\label{sec:pretopos-case}
Let us now write $\twocat{Pretop}$ for the $2$-category of
pretoposes, pretopos morphisms and all natural transformations. Like
before, every locally connected Grothendieck topos is a pretopos and
every inverse image functor is a pretopos morphism, so that we have a
forgetful $2$-functor
$\twocat{LCGTop}^\mathrm{op} \rightarrow \twocat{Pretop}$. 

\begin{Defn}
  \label{def:24}
  A \emph{locally connected classifying topos} for a pretopos $\C$ is
  a left biadjoint at $\C$ for the forgetful $2$-functor
  $\twocat{LCGTop}^\mathrm{op} \rightarrow \twocat{Pretop}$.
\end{Defn}

It is known that every small pretopos $\C$ has a locally connected
classifying topos. To see this, we factor the forgetful $2$-functor of
the preceding definition as the composite of the two forgetful $2$-functors
\begin{equation*}
  \twocat{LCGTop}^\mathrm{op} \rightarrow \twocat{GTop}^\mathrm{op}
\rightarrow \twocat{Pretop}
\end{equation*}
viewing a locally connected Grothendieck topos as a Grothendieck
topos, and a Grothendieck topos as a pretopos.
The second factor is well-known to have a left biadjoint at every
small pretopos $\C$, given by the topos of sheaves $\cat{Sh}(\C)$ for
the topology of finite jointly epimorphic families. On the other hand,
the first factor is known to have a left biadjoint given by the
\emph{locally connected coclosure} of~\cite{Funk1999The-locally}. It
follows that the composite has a left biadjoint at every small
pretopos.

One difficulty with the preceding argument is that the construction of
the locally connected coclosure in~\cite{Funk1999The-locally} is
inexplicit, relying at a crucial point on the adjoint functor theorem.
Our objective in this section is to give a concrete description of the
locally connected classifying topos of any small \emph{De Morgan}
pretopos.
The notion of De Morgan pretopos is an obvious generalisation of the
notion of \emph{De Morgan topos} described, for example in~\cite[\sec
D4.6]{Johnstone2002Sketches2}. In giving the definition, we recall a
\emph{pseudocomplement} of an element $a$ in a distributive lattice is
an element $\neg a$ which is disjoint from $a$, and is moreover the
maximal such element; i.e., such that $a \wedge b = \bot$ if and only
if $b \leqslant \neg a$.

\begin{Defn}
  \label{def:25}
  A distributive lattice $A$ is a \emph{Stone
    algebra}~\cite{Gratzer1957On-a-problem} if it admits all
  pseudocomplements and satisfies 
  $\neg a \vee \neg \neg a = \top$ for all $a \in A$. A pretopos $\C$
  is \emph{De Morgan} if each subobject lattice $\mathrm{Sub}_\C(X)$
  is a Stone algebra.
\end{Defn}

An equivalent characterisation of a De Morgan pretopos is as one in
which each inclusion of meet semi-lattices
$\mathrm{Sum}_\C(X) \rightarrow \mathrm{Sub}_\C(X)$ has a left adjoint
sending $A$ to $\neg \neg A$.
The relevance of the condition to our
investigations is isolated in the following result, whose significance
will become clear shortly. In its proof, we use the operation
$\exists_f \colon \mathrm{Sub}_\C(X) \rightarrow \mathrm{Sub}_\C(Y)$
of \emph{direct image} along a map $f \colon X \rightarrow Y$ of a
pretopos $\C$. This operation is left adjoint to pullback
$f^{-1} \colon \mathrm{Sub}_\C(Y) \rightarrow \mathrm{Sub}_\C(X)$ and
satisfies the \emph{Beck--Chevalley} and \emph{Frobenius} conditions;
see~\cite[\sec A1.3]{Johnstone2002Sketches}.

\begin{Prop}
  \label{prop:21}
  If $\C$ is a De Morgan pretopos, then $\uf_\C$ satisfies the right
  Ore condition: that is, each cospan in $\uf_\C$ as in the solid part
  of the following diagram can be completed to a commuting square as
  shown:
  \begin{equation}\label{eq:52}
    \cd{
      (Z, \W) \ar@{-->}[r]^-{[g_2]} \ar@{-->}[d]_-{[g_1]}& (X_2, \U_2)
      \ar[d]^-{[f_2]} \\
      (X_1, \U_1) \ar[r]^-{[f_1]} & (Y, \V)\rlap{ .}
    }
  \end{equation}
\end{Prop}

\begin{proof}
  Since every map in $\uf_\C$ factors as an isomorphism followed by
  the equivalence class of a \emph{total} map, we lose no generality
  in assuming that the $f_i$'s in~\eqref{eq:52} are total. We can
  therefore form their pullback
  $g_1 \colon X_1 \leftarrow Z \rightarrow X_2 \colon g_2$ in $\C$,
  and consider the subset
  $\F \subseteq \mathrm{Sum}_\C(Z)$ given as the upward closure of
  \begin{equation}\label{eq:40}
 \{ g_1^{-1}(U_1) \cap g_2^{-1}(U_2) : U_1 \in \U_1,
 U_2 \in \U_2 \}\rlap{ .}
 \end{equation}

 This subset is easily a filter on $\mathrm{Sum}_\C(Z)$, and we claim
 it is a \emph{proper} filter; thus, for any $U_1 \in \U_1$ and
 $U_2 \in \U_2$ we must show that
 $g_1^{-1}(U_1) \cap g_2^{-1}(U_2) \neq \bot$. Now, by Frobenius,
 Beck--Chevalley, and Frobenius we have
 \begin{align*}
   \exists_{f_1g_1}(g_1^{-1}(U_1) \cap g_2^{-1}(U_2)) & =
   \exists_{f_1}(U_1
   \cap \exists_{g_1}(g_2^{-1}(U_2))) \\
   & = \exists_{f_1}(U_1 \cap f_1^{-1}(\exists_{f_2}(U_2))) \\ &=
   \exists_{f_1}(U_1) \cap \exists_{f_2}(U_2)\rlap{ ,}
 \end{align*}
 and so, since direct image preserves and reflects $\bot$, we must
 equally show that
 $\exists_{f_1}(U_1) \cap \exists_{f_2}(U_2) \neq \bot$. If we set
 $V_i = \neg \neg \exists_{f_i}(U_i)$ then, by standard properties of
 pseudocomplementation, this is in turn equivalent to showing that
 $V_1 \cap V_2 \neq \bot$.
 Since $\C$ is De Morgan, we have $V_i \in \mathrm{Sum}_\C(Y)$;
 moreover, $U_i \in \U_i$ and
 $U_i \subseteq f_i^{-1}(\exists_{f_i}(U_i)) \subseteq f_i^{-1}(V_i)$
 implies $f_i^{-1}(V_i) \in \U_i$, and so $V_i \in \V$ since
 $(f_i)_!(\U_i) = \V$. Since $\V$ is an ultrafilter, we
 conclude that $V_1 \cap V_2 \neq \bot$ as desired.

 This proves that~\eqref{eq:40} generates a \emph{proper} filter $\F$.
 By the Boolean prime ideal theorem, we can extend this to an
 ultrafilter $\W \in \beta Z$, which by construction satisfies
 $\U_i \subseteq (g_i)_!(\W)$ for $i = 1,2$, and so
 $\U_i = (g_i)_!(\W)$ (since both sides are ultrafilters). We have
 thus completed~\eqref{eq:40} to a commuting square as
 desired.\end{proof}

The key to constructing the locally connected classifying topos of a
small De Morgan pretopos is the following standard result on geometric
morphisms into sheaf toposes proved, for example,
in~\cite[Lemma~VII.7.3]{Mac-Lane1994Sheaves}. In the statement, we
write $\twocat{CovFlat}(\A, \E)$ for the category of flat functors $\A
\rightarrow
\E$ which are also \emph{cover-preserving}, in the sense of sending
covers to jointly epimorphic families.

\begin{Prop}
  \label{prop:22}
  Let $\A$ be a small site and
  $i \colon \cat{Sh}(\A) \rightarrow [\A^\mathrm{op}, \cat{Set}]$ the
  associated inclusion of toposes. Under the
  equivalence~\eqref{eq:53}, a geometric morphism $\E \rightarrow
  [\A^\mathrm{op}, \cat{Set}]$ factors
  through $i$ just when the corresponding flat functor $\A \rightarrow
  \E$ is
  cover-preserving. Consequently,~\eqref{eq:53} restricts back to an
  equivalence
  \begin{equation}\label{eq:50}
    \twocat{GTop}(\E, \cat{Sh}(\A)) \simeq \twocat{CovFlat}(\A, \E)\rlap{ .}
  \end{equation}
\end{Prop}

The locally connected classifying topos of the small De Morgan
pretopos $\C$ will be obtained as a topos of sheaves on $\uf_\C$ for a
suitable Grothendieck topology, and its universal property verified
via a chain of pseudonatural equivalences
$\twocat{LCGTop}(\E, \cat{Sh}({\uf_\C})) \simeq
\twocat{CovFlat}(\uf_\C, \E) \simeq \twocat{Pretop}(\C, \E)$, each
of whose terms is a restriction of the corresponding term
in~\eqref{eq:49}.

Since a pretopos morphism out of $\C$ is a
lextensive functor which also preserves regular epimorphisms, the
topology on $\uf_\C$ must be chosen so that, under the equivalence
$\twocat{Flat}(\uf_\C, \E) \simeq \twocat{Lext}(\C, \E)$
of~\eqref{eq:49}, the cover-preserving functors to the left correspond
to the regular-epimorphism-preserving ones to the right.

We now describe such a topology, specifying it in terms of a
\emph{coverage}~\cite[Definition~C2.1.1]{Johnstone2002Sketches2}; this
involves assigning to each object $X$ a set of covering families $(f_i
\colon X_i \rightarrow X \mid i \in
I)$ satisfying the stability property:
\begin{itemize}
\item [(C)] For any cover
  $(f_i \colon X_i \rightarrow X \mid i \in I)$ and any map
  $g \colon Y \rightarrow X$ in $\A$, there is a cover
  $(h_j \colon Y_j \rightarrow Y \mid j \in J)$ such that each $gh_j$
  factors through some $f_i$.
\end{itemize}

\begin{Prop}
  \label{prop:23}
  Let $\C$ be a pretopos. There is a coverage on
  $\uf_\C$ for which a typical cover of the object $(Y, \V) \in
  \uf_\C$ is of the form 
  \begin{equation}\label{eq:22}
    \spn{f,\V} \defeq \big(\,[f] \colon (X, \U) \rightarrow (Y, \V)\, \mid\,\U \in
    \beta X, f_!(\U) = \V\,\big)
  \end{equation}
  for any $f \colon X \rightarrow Y$ whose image $\im f
  \rightarrowtail Y$ is (a coproduct injection and) in $\V$.
\end{Prop}

\begin{proof}
  We must verify condition (C). So consider $\spn{f,
    \V}$ as above and a map $[g] \colon (Y', \V') \rightarrow (Y,
  \V)$ in $\uf_\C$ defined on some $m \colon V' \rightarrowtail
  Y'$ in $\V'$. We first pull back $f$ along $g$ in
  $\C$ as in the left-hand square below, and now define $f' = mq \colon X' \rightarrow
  Y'$. By assumption, $\im f \in \V$; since
  $g$ is continuous and image factorisations are pullback-stable, it
  follows that $\im f' \in \V'$. Moreover, for each $[f'] \colon (X',
  \U') \rightarrow (Y', \V')$ in $\spn{f',
    \V'}$, the composite $[gf']$ factorises through a map in $\spn{f,
    \V}$ as to the right in:
\begin{equation*}
  \cd{
    {X'} \pullbackcorner \ar[r]^-{p} \ar[d]_{q} &
    {X} \ar[d]^{f} & &
    {(X', \U')} \ar@{-->}[r]^-{[p]} \ar[d]_{[f']} &
    {(X, p_! \U')} \ar[d]^{[f]} \\
    {V'} \ar[r]^-{g} &
    {Y} & &
    {(Y', \V')} \ar[r]^-{[g]} &
    {(Y, \V)}\rlap{ .}
  }  
\end{equation*}
This proves that the covers do indeed satisfy condition (C).
\end{proof}

We write $\cat{Sh}(\uf_\C)$ for the
topos of sheaves on $\uf_\C$ for this coverage.
\begin{Thm}
  \label{thm:12}
  Let $\C$ be a small De Morgan pretopos. The topos
  $\cat{Sh}({\uf_\C})$  is a locally connected classifying topos for $\C$, and is
  itself De Morgan.
\end{Thm}
\begin{proof}
  We begin by showing that $\cat{Sh}(\uf_\C)$ is locally connected and
  De Morgan. Since $\C$ is De Morgan, we know by
  Proposition~\ref{prop:21} that $\uf_\C$ satisfies the right Ore
  condition, and so by~\cite[Examples
  C3.3.11(a)]{Johnstone2002Sketches2}
  and~\cite[Corollary~2.8]{Caramello2009De-Morgan}, the sheaf topos
  $\cat{Sh}(\uf_\C)$ will be both locally connected and De Morgan so
  long as every covering family $\spn{f, \V}$ as in~\eqref{eq:22} is
  non-empty. Thus, given $\V \in \beta Y$ and
  $f \colon X \rightarrow Y$ in $\C$ with $\im f \in \V$, we must show
  that there exists an ultrafilter $\U \in \beta X$ with
  $f_!(\U) = \V$. Much as in Proposition~\ref{prop:21}, we consider
  the subset $\F \subseteq \mathrm{Sum}_\C(X)$ given as the
  upwards-closure of
  \begin{equation}\label{eq:55}
    \{f^{-1}(V) : V \in \V\}\rlap{ .}
  \end{equation}
  Like there, $\F$ is a filter which we claim is moreover
  \emph{proper}. Indeed, if $\bot = f^{-1}(V)$ for some
  $V \in \mathrm{Sum}_\C(Y)$, then also
  $\bot = \exists_f(f^{-1}(V) \cap \top) = V \cap \im f$ by Frobenius;
  whence $V \notin \V$ since $\im f \in \V$. Like before, we can now
  use the Boolean prime ideal theorem to find an ultrafilter
  $\U \subseteq \mathrm{Sum}_\C(X)$ extending $\F$ which, by
  construction, will satisfy $\V \subseteq f_!(\U)$ and hence (since
  both are ultrafilters) $\V = f_!(\U)$.

  So $\cat{Sh}(\uf_\C)$ is locally connected and De Morgan; it remains
  to verify the classifying property, for which we must exhibit
  equivalences
  $\twocat{LCGTop}(\E, \cat{Sh}({\uf_\C}), \cat{Set}]) \simeq
  \twocat{Pretop}(\C, \E)$, pseudonaturally in $\E$. As discussed
  above, these will be obtained by composing pseudonatural
  equivalences:
  \begin{equation}\label{eq:51}
    \twocat{LCGTop}(\E, \cat{Sh}({\uf_\C})) \xrightarrow{\simeq}
    \twocat{CovFlat}(\uf_\C, \E) \xrightarrow{\simeq} \twocat{Pretop}(\C, \E)
  \end{equation}
  of which the first is~\eqref{eq:50}, and the second is obtained by
  restricting the right-hand equivalence
  $\twocat{Flat}(\uf_\C, \E) \simeq \twocat{Lext}(\C, \E)$
  of~\eqref{eq:49}. The only point to check is that the
  cover-preserving functors to the left of this latter equivalence
  correspond to the regular-epimorphism-preserving ones to the right.
  
  So suppose given a covering family $\spn{f,\V}$ as in~\eqref{eq:22}.
  We may form the image factorisation
  $f = me \colon X \twoheadrightarrow \im f \rightarrowtail Y$, and
  since by assumption $\im f \in \V$, we conclude that
  \begin{equation*}
    \spn{f, \V} = \big(\,(X, \U) \xrightarrow{[e]} (\im f, \res \V {\im f}) \xrightarrow{[m]}
    (Y, \V)\,\big)_{\U \in \beta X, e_!(\U) = \res \V {\im f}}\rlap{ .}
  \end{equation*}
  Since $[m]$ is invertible in $\uf_\C$, this family will be sent to a
  jointly epimorphic one just when $\spn{e, \res \V {\im f}}$ is;
  whence a functor $A \colon \uf_\C \rightarrow \E$ preserves all
  covers just when it preserves ones $\spn{f, \V}$ as in~\eqref{eq:22}
  with $f$ a \emph{regular epimorphism}. This is equally to say that,
  for each $f \colon X \twoheadrightarrow Y$ and each $\V \in \beta Y$,
  the map to the left in:
  \begin{equation*}
    \sum_{\substack{\U \in \beta X\\f_!(\U) = \V}} A(X,\U) \rightarrow
    A(Y, \V) \qquad\qquad \sum_{\U \in \beta X} A(X,\U) \rightarrow \sum_{\V
      \in \beta X} A(Y, \V)
  \end{equation*}
  obtained by copairing the maps
  $A([f]) \colon A(X,\U) \rightarrow A(Y,\V)$ is an epimorphism in
  $\E$. Summing these left-hand maps over all $\V \in \beta Y$ and
  using infinite extensivity of $\E$, this is equally the condition
  that, for each $f \colon X \twoheadrightarrow Y$ in $\C$, the map
  right above is an epimorphism. Since this map is the value at $f$ of
  the functor $\int\!A \colon \C \rightarrow \E$ corresponding to $A$
  under~\eqref{eq:13}, this completes the proof.
\end{proof}

\begin{Rk}
  \label{rk:10}
  As before, chasing the identity geometric morphism
  $\cat{Sh}(\uf_\C) \rightarrow \cat{Sh}(\uf_\C)$ through this proof
  shows that the universal pretopos map
  ${\eta \colon \C \rightarrow \cat{Sh}({\uf_\C})}$ is the image
  under~\eqref{eq:13} of the composite
  $ay \colon \uf_\C \rightarrow \cat{Sh}({\uf_\C})$ of the Yoneda
  embedding and the sheafification functor. As such it is given by:
  \begin{equation*}
    \eta(X) = \textstyle\sum_{\U \in \beta X} ay_{(X, \U)}\rlap{ .}
  \end{equation*}
  Now, each object $ay_{(X, \U)}$ is \emph{connected} in
  $\cat{Sh}(\uf_\C)$, in the sense of having no non-trivial coproduct
  decomposition; and so it follows that
  $\cat{Sum}_{\uf_\C}(\eta X) = \P(\beta X)$ is the canonical
  extension of the Boolean algebra $\cat{Sum}_\C(X)$. This justifies
  the link drawn in the introduction to this section between locally
  connected classifying toposes, and canonical extension for Boolean
  algebras.
\end{Rk}

\begin{Rk}
  \label{rk:7}
  We remarked above that the $2$-functor
  $\twocat{GTop}^\mathrm{op} \rightarrow \twocat{Pretop}$ has a
  left biadjoint at every small pretopos given by the topos
  $\cat{Sh}(\C)$ of sheaves on $\C$ for the topology of finite jointly
  epimorphic families. The toposes arising in this way are commonly
  known as \emph{coherent} toposes; moreover,
  by~\cite[Theorem~3.11]{Caramello2009De-Morgan}, the coherent topos
  associated to a small De Morgan pretopos is itself De Morgan. Given
  this, another way of seeing Theorem~\ref{thm:12} is as giving an
  explicit construction of the \emph{locally connected
    coclosure}~\cite{Funk1999The-locally} of any coherent De Morgan
  topos.

  One may reasonably ask if we have a similar explicit construction
  upon dropping the qualifier ``coherent''. The answer is yes, so long
  as we assume that every cardinal is smaller than some strongly
  compact cardinal. In this case, for any De Morgan Grothendieck topos
  $\E$, we can find a strongly compact cardinal $\kappa$ such that
  $\E$ is the free completion of a small De Morgan \emph{$\kappa$-ary
    pretopos}---that is, a pretopos with pullback-stable $\kappa$-small coproducts. We
  can thus reduce the problem to constructing the locally connected
  classifying topos of a small $\kappa$-ary De Morgan pretopos; and we
  can do this by tracing through the definitions and results of this
  paper replacing everywhere finite coproducts by $\kappa$-small
  coproducts. The main change, as in~\cite{Borger1980A-Characterization}, is
  that we must replace ultrafilters by \emph{$\kappa$-complete
    ultrafilters}---ones closed under $\kappa$-small intersections.
  The assumption of strong compactness of $\kappa$ is needed in
  the proofs of Proposition~\ref{prop:21} and Theorem~\ref{thm:12},
  where we are now required to extend the $\kappa$-complete
  filters~\eqref{eq:40} and~\eqref{eq:55} to $\kappa$-complete
  ultrafilters.
\end{Rk}

\subsection{Relation to toposes of types}
\label{sec:relation-topos-types}

In the introduction to this section, we have already discussed the
analogies between our construction of a locally connected classifying
topos, and Makkai's topos of types from~\cite{Makkai1981The-topos}. We
now provide a more detailed technical comparison between these
constructions and other similar constructions in the literature.

The earliest ``topos of types'' in fact predates Makkai, appearing in
Joyal and Reyes'~\cite{Joyal1978Forcing}. Given a pretopos $\C$, a
\emph{prime filter} on $X \in \C$ is a prime filter in the
distributive lattice $\mathrm{Sub}_\C(X)$; these comprise the objects
of a category $\P\F_\C$ of prime filters in $\C$ defined similarly to
$\uf_\C$. Endowing $\P\F_\C$ with the obvious analogy of the topology
of Proposition~\ref{prop:23} yields~\cite{Joyal1978Forcing}'s
\emph{topos of existential types}. No universal property is described,
but the formula~\eqref{eq:17} appears on p.11 of~\emph{ibid}.

Makkai's topos of types from~\cite{Makkai1981The-topos} is a topos of
sheaves on the same category $\P\F_\C$, but for a different topology,
and he exhibits it as the ``classifying prime-generated topos for
$\mathfrak p$-models of $\C$''. As we have already said, a
Grothendieck topos $\E$ is prime-generated if each subobject lattice
is completely distributive algebraic, while a pretopos morphism
$F \colon \C \rightarrow \E$ into a prime-generated topos is said to
be a \emph{$\mathfrak p$-model} if for every prime filter
$\mathfrak p$ on $\cat{Sub}_\C(X)$ and every
$f \colon X \rightarrow Y$ we have
$\exists_f(\bigcap_{A \in \mathfrak p} FA) = \bigcap_{A \in \mathfrak
  p} F(\exists_f A)$ in $\mathrm{Sub}_\E(FY)$. The classifying
property of the topos of types $\tau(\C)$ is given by equivalences
\begin{equation}\label{eq:21}
  \twocat{PGTop}(\E, \tau(\C)) \simeq
  \twocat{pPretop}(\C, \E)
\end{equation}
where to the left we have the category of geometric morphisms between
prime-generated toposes whose inverse image functors preserve all
intersections, and to the right we have the category of
$\mathfrak p$-models. In establishing this equivalence, the
formula~\eqref{eq:17} again appears; see the bottom of p.164
of~\emph{ibid}. In model-theoretic terms, the condition of being a
$\mathfrak{p}$-model is a \emph{saturation} condition; Makkai states
this already in~\cite{Makkai1981The-topos}, and the point is followed
up in~\cite{Butz2004Saturated}, and exploited
in, among other places,~\cite{Moerdijk1995A-model, Eliasson2003Ultrasheaves}.

The other main ``topos of types'' in the literature is Pitts'
\emph{topos of filters} $\Phi(\C)$ of a pretopos $\C$. Introduced
in~\cite{Pitts1983An-application}, this is the topos of sheaves on the
category $\F_\C$ of \emph{all}---not necessarily prime---filters of
subobjects, for the topology whose covers are the finite jointly
epimorphic families. The universal property of $\Phi(\C)$ was given
in~\cite{Magnan2000Le-topos} by analogy with $\tau(\C)$: it is the
``classifying completely distributive topos for $\mathfrak{f}$-models
of $\C$''. Here, a completely distributive topos is one whose
subobject lattices are completely distributive, and an
$\mathfrak{f}$-model is like a $\mathfrak{p}$-model, but with
arbitrary filters replacing prime ones.

We conclude this discussion by comparing the universal
characterisation~\eqref{eq:21} of Makkai's topos of types and our
Theorem~\ref{thm:12}. To the left of the equivalence, our theorem
replaces ``prime-generated'' by ``locally connected'' and moreover
relaxes the condition of intersection-preservation on morphisms. What
permits this relaxation is the fact that we only care about
intersections of \emph{coproduct summands}, and any inverse image
functor between locally connected toposes preserves these. To the
right of the equivalence, we drop the $\mathfrak p$-model condition.
This is to do with the fact that our choice of topology is analogous
to Joyal and Reyes'~\cite{Joyal1978Forcing} rather than
Makkai's~\cite{Makkai1981The-topos}. If one modifies Makkai's topos of
types to use Joyal and Reyes' topology, then one can also drop the
$\mathfrak p$-model condition; however, the result is then no longer a
prime-generated topos, and so it is unclear what an appropriate
universal property would be. The final difference we note is that
Makkai's equivalence works for arbitrary pretoposes $\C$, while ours
works only for \emph{De Morgan} pretoposes; this extra condition seems
to be necessary to ensure that the topos of sheaves we form is indeed
locally connected.

Asides from these technical distinctions, we would raise one further
point. In this paper, we have striven to make the constructions we
give as unavoidable as possible. The category $\uf_\C$ is forced upon
us once we are interested in finite-coproduct-preserving functors out
of $\C$; adding finite-limit-preservation leads us to consider also
flatness; and finally, once we add regular-epimorphism-preservation,
we are led inevitably to the given topology on $\uf_\C$. Everything
else is a matter of making the details match up\footnote{Though at
  this stage we have no satisfactory explanation for the requirement
  of De Morganness.}. In future work, we intend to see whether our
main results can be adapted to the prime filter setting, and if, on
doing so, they provide a treatment of Makkai's topos of types in the
same spirit.

\end{document}